\documentclass[a4paper,11pt]{article}

\usepackage[a4paper,margin=2.5cm]{geometry}

\usepackage{graphicx}
\usepackage{parskip}
\usepackage{amssymb, amsmath}
\usepackage[amsmath,hyperref,thmmarks]{ntheorem}
\usepackage[utf8]{inputenc}
\usepackage[english]{babel}
\usepackage{enumerate}
\usepackage{color,caption}
\usepackage{mathtools}

\usepackage{float}

\usepackage[
  backend=biber,
  style=numeric
]{biblatex}

\usepackage{hyperref}

\numberwithin{equation}{section}

\def\cF{{\mathcal F}}

\def\E{{\mathbb E}}
\def\F{{\mathbb F}}
\def\H{{\mathbb H}}

\def\L{{\mathbb L}}

\def\P{{\mathbb P}}

\def\R{{\mathbb R}}

\def\ind#1{{\bf 1}_{\left\{#1\right\}}}
\def\ind2#1{{\bf 1}_{#1}}

\def\pp{\partial}

\def \mP{\mathbb{P}}
\def \R{\mathbb{R}}
\def \E {\textrm{E}}

\newtheorem{theorem}{Theorem}[section]
\newtheorem{proposition}[theorem]{Proposition}
\newtheorem{lemma}[theorem]{Lemma}

\newenvironment{proof}{\textit{Proof}.}
{\hfill$\square$}
\newenvironment{idea_proof}{\textit{Idea of proof}.}
{\hfill$\square$}
\newtheorem{asmp}[theorem]{Assumption}
\newtheorem{rem}[theorem]{Remark}

\newtheorem{alg}{Algorithm}[section]

\newcommand{\1}{\mathbf{1}}

\begin{document}

\date{\small \today}

\title{A Martingale approach to continuous Portfolio Optimization under CVaR like constraints}

\author{Jérôme Lelong \thanks{Univ. Grenoble Alpes, CNRS, Grenoble INP, LJK, 38000 Grenoble, France. \texttt{jerome.lelong@univ-grenoble-alpes.fr}} \ \ Véronique Maume-Deschamps \thanks{Universite Claude Bernard Lyon 1, CNRS, Ecole Centrale de Lyon, INSA Lyon, Université Jean Monnet, ICJ UMR5208,
69622 Villeurbanne, France. \texttt{veronique.maume-deschamps@univ-lyon1.fr}} \\  William Thevenot \thanks{Universite Claude Bernard Lyon 1, CNRS, Ecole Centrale de Lyon, INSA Lyon, Université Jean Monnet, ICJ UMR5208,
        69622 Villeurbanne, France. and Risk Knowledge team at SCOR SE, Paris, France \texttt{thevenot@math.univ-lyon1.fr}} }

\maketitle

\begin{abstract}
    We study a continuous-time portfolio optimization problem under an explicit constraint on the Deviation Conditional Value-at-Risk (DCVaR), defined as the difference between the CVaR and the expected terminal wealth. While the mean-CVaR framework has been widely explored, its time-inconsistency complicates the use of dynamic programming. We follow the martingale approach in a complete market setting, as in Gao et al.~\cite{gao2017dynamic}, and extend it by retaining an explicit DCVaR constraint in the problem formulation.
    
    The optimal terminal wealth is obtained by solving a convex constrained minimization problem. This leads to a tractable and interpretable characterization of the optimal strategy.
  \end{abstract}



\section*{Introduction}

Managing the risk associated with rare but high-impact events is a fundamental challenge in finance, in the insurance and reinsurance industries, as it requires designing portfolios that remain resilient under such conditions. Historically, variance has served as the canonical risk measure following the seminal work of H. Markowitz in 1952~\cite{markowitz1952modern}. However, its symmetric treatment of gains and losses makes it ill-suited for capturing tail risk. Alternative measures, such as Value-at-Risk (VaR) and Conditional Value-at-Risk (CVaR), have been introduced to address this issue. Among them, CVaR has become a standard due to its coherence, as introduced by P. Artzner, F. Delbaen, J.-M. Eber and D. Heath in 1999~\cite{Artzner1999:11}, and its focus on worst-case losses.

In this context, portfolio optimization under CVaR constraints, known as the mean-CVaR framework, gained prominence after the work of R.T. Rockafellar and S. Uryasev in 2000~\cite{Uryasev2000:7} and 2002~\cite{Uryasev2002:9}. A major challenge in this setting is the time-inconsistency of CVaR, as discussed by Artzner et al. in 2007~\cite{artzner2007coherent} and later formalized by A. Shapiro in 2009~\cite{shapiro2009time}, which complicates the use of classical stochastic control techniques. Several solutions have been proposed to address this issue. In discrete time, M. Strub, D. Li, X. Cui and J. Gao in 2019~\cite{strub2019discrete} reformulated the mean-CVaR problem as a family of expected utility maximization problems with piecewise linear utility functions, allowing for analytical tractability and induced time-consistency. In continuous time, J. Gao, K. Zhou, D. Li and X. Cao in 2017~\cite{gao2017dynamic} adopted a martingale method in a complete market, transforming the dynamic problem into a static optimization under the risk-neutral measure. Their approach yields explicit solutions, but relies on strong technical assumptions, including bounded terminal wealth.

An alternative route was explored by C. Miller and I. Yang in 2017~\cite{miller2017optimal}, who formulated the mean-CVaR control problem in continuous time using the framework of bilevel optimization. Their method does not rely on a complete market structure and provides an elegant way to handle time-inconsistency through nested optimization, though it typically requires solving a sequence of value function problems.

This paper follows the structural approach of Gao et al.~\cite{gao2017dynamic}, within a complete market setting, while introducing two key differences. First, we replace the CVaR with the Deviation-CVaR (DCVaR), defined as the difference between CVaR and the expected value of terminal wealth. Although less common, DCVaR behaves similarly to a deviation measure and retains both the coherence and tail sensitivity of CVaR, making it appealing in a dynamic context. The formal properties of DCVaR as a generalized deviation measure were characterized by R.T. Rockafellar, S. Uryasev and M. Zabarankin in 2006~\cite{rockafellar2006generalized}. Second, we impose an explicit constraint on the DCVaR, rather than incorporating it into the objective through a Lagrangian relaxation. This preserves the interpretability of the constraint and maintains a clear separation between risk control and performance.

The optimal control is derived using the martingale method under the complete market assumption, which leads to an explicit or semi-analytical characterization of the optimal strategy. One important limitation, as in Gao et al.~\cite{gao2017dynamic}, is the impossibility of enforcing bounds on the control process. While this restricts certain practical applications, it follows naturally from the structure of the problem and facilitates analytical tractability.

We formulate a continuous-time portfolio optimization problem with an explicit DCVaR constraint. We solve the problem using martingale techniques and characterize the optimal control. We analyze the structure and limitations of the solution within this framework.

The paper is organized as follows. In Section~\ref{sec:problem-settings}, we present the problem setting and formal model. Section~\ref{sec:resolution} details the resolution via martingale methods, including the formulation of the optimal terminal wealth and the Lagrangian relaxation. Section~\ref{sec:sse_determin} applies the results to a deterministic multidimensional Black-Scholes market and studies the asymptotic behavior of the optimal solution. In Section~\ref{sec:numerics}, we provide numerical illustrations and visualize the efficient frontier. Technical proofs are provided in the Appendix.

\section{Problem settings}\label{sec:problem-settings}
We consider a financial market consisting by $n$ risky assets and one risk-free asset, which can be traded continuously over the time horizon $[0,T]$.

All sources of uncertainty are modeled on a complete filtered probability space
$(\Omega,\mathcal{A},\mathbb{P},\mathbb{F}=\{\mathcal{F}_t\}_{0\le t\le T}),$
supporting an $n$‑dimensional Brownian motion
$W_t=(W^1_t,\dots,W^n_t)^\top,$
whose completed natural filtration is $\mathbb{F}$, and for all $i\neq j$ the processes
$W^i=(W^i_t)_{0\le t\le T}\quad\text{and}\quad W^j=(W^j_t)_{0\le t\le T}$
are independent.

Let $$\H^2_d=\{(X_t)_{0\le t\le T}\mid X \text{ is an }\R^d \text{-valued }\mathbb F\text{-adapted process and }\E\bigl[\int_0^T|X_t|^2\,dt\bigr]<\infty\},$$ and $$\L^2(\Omega,\mathcal{F}_T,\mathbb{P})=\{X\text{ is }\mathcal{F}_T\text{-measurable random variable},\ \E[|X|^2]<\infty\}.$$
We will specify the dimension of $\H^2$ only when necessary.

The price process $S_t^0$ of the risk-free asset evolves according to the ordinary differential equation:
\begin{align*}
  dS_t^0 &= r_t S_t^0\, dt, \quad t \in [0,T], \\
  S_0^0 &= s_0 > 0,
\end{align*}
where $(r_t)_{0\le t\le T}$ is a scalar-valued, $\F$-adapted stochastic process representing the risk-free interest rate. Assume $r$ is $\mathbb{F}$-predictable with $\int_0^T |r_s|\,ds<\infty$ a.s., and that there exist finite constants $m_r\le 0\le M_r$ such that $m_r \le r_t \le M_r$ a.s. for all $t\in[0,T]$. Define the discount factor
\[
D_t := \exp\!\Bigl(-\int_0^t r_s\,ds\Bigr), \qquad t\in[0,T],
\]
which is well defined and satisfies $0 < D_t \le e^{ct}$ for all $t\in[0,T]$. Consequently, $D_t \in L^p$ for every $p\ge 1$ and $t\in[0,T]$, i.e., $\mathbb{E}[D_t^p]<\infty$.

The price processes $(S_t^i)_{0\le t\le T}$ of the $n$ risky assets follow the system of stochastic differential equations:
\begin{align*} \label{eq:asset_dynamics}
  dS_t^i &= S_t^i \left( \mu_t^i\, dt + \sum_{j=1}^n \sigma_t^{ij}\, dW_t^j \right), \quad t \in [0,T], \quad i = 1, \dots, n, \\
  S_0^i &= s_i > 0, \quad i = 1, \dots, n,
\end{align*}

where $(\mu_t^i)_{0\le t\le T}$ and $(\sigma_t^{ij})_{0\le t\le T}$ represent the drift (expected return) and volatility coefficients, respectively. We assume that all $\mu^i$ and $\sigma^{ij}$ are uniformly bounded and $\F$-adapted scalar-valued stochastic processes.

Furthermore, we assume that the volatility matrix $\sigma_t := \{\sigma_t^{ij}\}_{i,j=1}^n \ , {0\le t\le T},$ satisfies the following uniform non-degeneracy condition:
\begin{align*}
  \sigma_t\sigma^{\top}_t \succ \epsilon I, \quad \text{for all } t \in [0,T], \text{ a.s.},
\end{align*}
for some constant $\epsilon > 0$. Under these assumptions, the market model is complete.

An investor with initial wealth $w_0$ enters the market at time $0$ and continuously allocates its wealth among the $n$ risky assets and the risk-free asset over the time horizon $[0,T]$. Let $(V_t)_{0\le t\le T}$ denote the total wealth process of the investor.\\

We denote the portfolio process by $\beta = (\beta^1,\dots,\beta^n)^\top \in \H^2_n$, where for each $i=1,\dots,n$, $(\beta^i_t)_{0\le t\le T}$ represents the number of shares held in asset $i$. We also define $\beta^0 \in \H^2_1$ the number of shares held of risk-free asset. Since we do not consider transaction costs during the investment period, the investor's wealth process $V$ satisfies the following stochastic differential equation (SDE) $\{ dV_t = \beta_t^0 dS_t^0 + \beta_t^{\top} dS_t; V_0 = w_0 \}$ which leads to:
\begin{align*}
    \begin{dcases}
    dV_t = \beta_t^0 r_t dt + \sum_{i=1}^n \beta_t^i S_t^i \left( \mu_t^i dt + \sum_{j=1}^n \sigma_t^{ij}\, dW_t^j \right), \\
    V_0=w_0,
    \end{dcases}
\end{align*}
and finally
\begin{align}
    \begin{dcases}
    dV_t = \left(r_tV_t+\sum_{i=1}^n \beta_t^i S_t^i b_t^i \right)dt + \sum_{i,j=1}^n\beta_t^i S_t^i \sigma_t^{ij}\, dW_t^j, \\
    V_0=w_0,
    \end{dcases}
    \label{def_wealth}
\end{align}

where $(b_t)_{0\le t\le T}$ is the excess return defined by
\begin{align*}
b_t:=\left(
                 \begin{array}{cccc}
                   \mu_t^1-r_t &  \mu_t^2-r_t & \cdots & \mu_t^n-r_t
                 \end{array}
               \right)^{\top}.
\end{align*}

We seek to solve the following optimization.

\begin{equation}
  \begin{aligned}
    &~~\min_{\beta \in \H^2 }~~-\mathbb{E}[V_T] \\
  \textrm{Subject to}~&~\begin{dcases}
                          \text{DCVaR}_{\kappa}(-V_T) \leq K, \\
                          \textrm{\{$V_{(\cdot)}$,$\beta_{(\cdot)}$\} statisfies (\ref{def_wealth}) }, \\
                          0 \leq V_T \leq B,
                        \end{dcases}
  \end{aligned}\tag{$P_{DCVaR}$} \label{pb_P_DCVAR}\\
\end{equation}
with 
\begin{equation*}
    \text{DCVaR}_{\kappa}[-V_T] = \mathbb{E}[V_T] + \text{CVaR}_{\kappa}[-V_T],
\end{equation*}
and
\begin{equation*}
    \text{CVaR}_{\kappa}[-V_T] = VaR_\kappa(-V_T) + (1 - \kappa)^{-1}\mathbb{E} \left[(-V_T - VaR_\kappa(-V_T))^+ \right]. \\
\end{equation*}

We make use of the characterization of CVaR from \cite{Uryasev2002:9}, where the Conditional Value-at-Risk is defined as the optimal value of a convex optimization.
\begin{align*}
    \text{DCVaR}_{\kappa}[-V_T] &= \mathbb{E}[V_T] + \min_{\alpha} \left(\alpha + (1 - \kappa)^{-1}\mathbb{E} \left[(-V_T - \alpha)^+ \right]\right), \\
    &= \min_{\alpha} \mathbb{E} \left[V_T + \alpha + (1 - \kappa)^{-1} (-V_T-\alpha)^+) \right]
\end{align*}
Keeping in line with the results from \cite{Uryasev2002:9}, the Equation \eqref{pb_P_DCVAR} can now be reformulated as follows
\begin{equation}
  \begin{aligned}
    &~~\min_{(\alpha,\beta)\in \mathbb{R}^ \times \H^2 }~~-\mathbb{E}[V_T]\\
  \textrm{Subject to}~&~\begin{dcases}
    \mathbb{E} \left[V_T+ \alpha + (1 - \kappa)^{-1} (-V_T-\alpha)^+) \right] \leq K \\
                          \textrm{\{$V_{(\cdot)}$,$\beta_{(\cdot)}$\} statisfies (\ref{def_wealth}) }, \\
                          0 \leq V_T \leq B,
                        \end{dcases}
  \end{aligned}\tag{$P_{DCVaR}-\text{Equivalent}$} \label{pb_P_DCVAR_equivalent}\\
\end{equation}

Moreover, when the optimal value is attained we have $\alpha^* = \operatorname{VaR}_\kappa(-V_T^*)$. Given that $V_T^* \geq 0$, we deduce that $\alpha^* \leq 0$.

At this stage, we consider the optimization with a fixed parameter $\alpha \in \mathbb{R}^-$. The optimal value of $\alpha$ will be addressed later. We now aim to solve the Equation \eqref{pb_P_DCVAR_alpha}.

\begin{equation}
  \begin{aligned}
    &~~\min_{\beta\in \H^2 }~~-\mathbb{E}[V_T]\\
  \textrm{Subject to}~&~\begin{dcases}
    \mathbb{E} \left[V_T+ \alpha + (1 - \kappa)^{-1} (-V_T-\alpha)^+) \right] \leq K \\
                          \textrm{\{$V_{(\cdot)}$,$\beta_{(\cdot)}$\} statisfies (\ref{def_wealth}) }, \\
                          0 \leq V_T \leq B,
                        \end{dcases}
  \end{aligned}\tag{$P_{DCVaR}(\alpha)$} \label{pb_P_DCVAR_alpha} \\
\end{equation}

\section{Resolution} \label{sec:resolution}
We assume a complete market, so there is a unique equivalent martingale measure \(\widetilde\P\) under which the discounted risky‐asset prices are martingales. Let \(L = \tfrac{d\widetilde\P}{d\P}\) be its Radon–Nikodým derivative on \(\F\), with conditional expectation \(L_t = \E[L\mid\cF_t]\). By the multidimensional Girsanov theorem \cite{karatzas1998methods}, \(L_t\) admits the exponential‐martingale representation \(dL_t = L_t\,\theta_t^\top\,dW_t\), where \(\theta\) is an \(\F\)‐adapted \(\R^m\)‐valued process chosen so that \(\tilde W \) is a Brownian motion under \(\widetilde\P\) such that \(\tilde W_t = W_t + \int_0^t\theta_s\,ds\); one shows a.s.\ that \(\theta_t = \sigma_t^{-1}b_t\) for \(0\le t\le T\). Under our standing assumptions on the short rate $r$ (together with the boundedness of $\mu$ and the uniform non-degeneracy of $\sigma$), the market price of risk $\theta$ is uniformly bounded on $[0,T]$.

Any square‐integrable and \(\mathcal{F}_T\)\nobreakdash‑measurable random variable can be perfectly replicated by a self‑financing strategy whose terminal value equals this variable almost surely. The following proposition is based on Chapter~5 of Shreve’s \emph{Stochastic Calculus for Finance II} \cite{shreve2004stochastic} (pp.~231–232).\\

\begin{proposition}
  Let \(M\in L^2(\Omega,\mathcal{F}_T,\mathbb{P})\). Then there exist a unique \(\F\)\nobreakdash‑predictable processes
  \(\bigl(\beta^0_t\bigr)_{0\le t\le T}\) and \(\bigl(\beta_t\bigr)_{0\le t\le T}\)
  such that the self‑financing portfolio value process \((V_t)_{0\le t\le T}\) defined by
  \[
    dV_t \;=\;\beta^0_t\,dS^0_t \;+\;\beta_t^\top\,dS_t,
    \qquad
    V_0 \;=\;\E_{\widetilde\P}\!\Bigl[D_T M\Bigr] \;=\;\E_{\P}\!\Bigl[L_T D_T M\Bigr]
  \]
  satisfies
  \[
    V_T = M \quad\text{a.s.}
  \]\label{prop:martingale_change}
\end{proposition}

\begin{idea_proof}

We can defined $V_t = \E_{\widetilde\P}\left[ e^{-\int_{t}^{T}r_s\,ds}\,M | \cF_t \right]$ for $t \in [0,T]$, we have that $(D_t V_t)_{0\le t\le T}$ is a $\widetilde\P$-martingale. 
By the martingale representation theorem under the unique equivalent martingale measure \(\widetilde\P\), there exists a predictable process \(\psi\) such that
\[
  D_tV_t \;=\; V_0 + \int_0^t \psi_s^\top\,d\widetilde W_s.
\]
We now seek to express the strategy \(\bigl(\beta_t\bigr)_{0\le t\le T}\) in terms of \(\bigl(\psi_t\bigr)_{0\le t\le T}\).\\

Under the measure \(\widetilde\P\), the discounted asset prices $(D_tS_t)_{0\le t\le T}$ satisfy
\[
  d(D_tS_t)
  =D_tS_t\odot\bigl(b_t\,dt+\sigma_t\,dW_t\bigr)
  =D_tS_t\odot\bigl(\sigma_t\,d\widetilde W_t\bigr).
\]
For a self‑financing strategy $(\beta_t)_{0\le t\le T}$, therefore,
\[
  d(D_tV_t)
  =\beta_t^\top\,d(D_tS_t)
  =D_t\,( \beta_t\odot S_t)^\top\,\sigma_t\,d\widetilde W_t.
\]
On the other hand, by martingale representation
\[
  d(D_tV_t)
  =\psi_t^\top\,d\widetilde W_t.
\]
Comparison of the two yields
\[
  (\beta_t\odot S_t)^\top\,\sigma_t
  =\psi_t/D_t,
\]
and since $(\sigma_t)_{0\le t\le T}$ is a.s.\ invertible (market completeness), the risky‐asset holdings are
\[
  \beta^i_t = \frac{\bigl(\sigma_t^{-1}\psi_t\bigr)_i}{D_t S^i_t},\quad i=1,\dots,n,
\]
and the cash position is
\[
  \beta^0_t = D_t V_t -\sum_{i=1}^n\bigl(\sigma_t^{-1}\psi_t\bigr)_i.
\]

At \(t=0\), by definition
\[
  V_0 = \E_{\widetilde\P}\bigl[D_T\,M\bigr].
\]
Since \(L_T=\frac{d\widetilde\P}{d\P}\big|_{\F_T}\), this can be written under the historical measure as
\[
  V_0 = \E_{\P}\bigl[L_T\,D_T\,M\bigr].
\]

Thus, the processes \(\{\beta^0_t,\beta_t\}_{0\le t\le T}\), together with the above condition, uniquely determine the replicating strategy.
\end{idea_proof}

For the next part, define the state‐price density, for $t\in[0,T]$, by
\begin{equation*}
  Z_t = L_t\,D_t, \quad
  dZ_t = -Z_t\bigl(r_t\,dt + \theta_t^\top\,dW_t\bigr), \quad
  Z_0 = 1,
\end{equation*}
which yields the closed‐form
\begin{equation}
  Z_t = \exp\!\Bigl(-\!\int_0^t\bigl(r_s + \tfrac12\|\theta_s\|^2\bigr)\,ds \;-\;\int_0^t\theta_s^\top\,dW_s\Bigr). \label{eq:def_z}
\end{equation}

Now, instead of directly searching for a portfolio strategy to solve the original optimization, one can first construct a terminal payoff \(M\in L^2(\Omega,\mathcal{F}_T,\mathbb{P})\), and then reconstruct the corresponding strategy. The state-price density process \( Z_T = L_T D_T\) enables us to reformulate the optimization under the historical measure \( \mathbb{P} \) by replacing $V_T$ by $M \in L^2(\Omega,\mathcal{F}_T,\mathbb{P})$ and add the constraint $V_0 = \mathbb{E}_{\mathbb{P}}[Z_T M]$.

\subsection{Optimal terminal wealth formulation}

Using Proposition \ref{prop:martingale_change}, it suffices to determine the optimal terminal wealth $V_T$ for the Equation \eqref{pb_P_DCVAR_alpha}, which can be obtained by solving the following static optimization. We assume $\kappa>0.5$ (typically around $0.99$) and denoting the terminal value of $Z_T$ simply by $Z$.

\begin{align}
  \min_{M \in \L^2(\Omega,\mathcal{F}_T,\mathbb{P})} \quad & -\mathbb{E}[M], \label{PB_A}\\
  \text{s.t.} \quad & \mathbb{E} \left[M + \alpha + (1 - \kappa)^{-1} (-M-\alpha)^+) \right] \leq K, \label{c_DCVaR_K}\\
  & \mathbb{E}[Z M] = w_0, \label{c_ZV_equal_V0}\\
  & 0 \leq M \leq B. \nonumber
\end{align}

\subsection{Lagrange relaxation of problem}
We introduce Lagrange multipliers $\lambda \geq 0$ and $\eta \in \R$, respectively for constraint \eqref{c_DCVaR_K} and \eqref{c_ZV_equal_V0}. It leads to the Lagrange relaxation of the Equation \eqref{PB_A}
\begin{equation}
    \min_{M \in \L^2(\Omega,\mathcal{F}_T,\mathbb{P}), 0 \leq M \leq B} \mathbb{E} \left[ -M + \lambda \left( M + \alpha + (1 - \kappa)^{-1} (-M-\alpha)^+ - K \right) + \eta (Z M - w_0) \right]. \label{eq:lagg_relaxation_pb}
\end{equation}

We can exclude the case $\eta \leq 0$ based on the sensitivity interpretation of Lagrange multipliers provided in Section~5.6.3 of the book by Boyd and Vandenberghe~\cite{boyd2004convex}. Since $\omega_0$ represents the initial wealth of the strategy, increasing its value naturally leads to a better (i.e., lower) optimal cost. Therefore, the associated Lagrange multiplier $\eta$ must be nonnegative.\\

\begin{proposition}
  The optimal solution of \eqref{eq:lagg_relaxation_pb} is given by:
  \begin{align}
      M^* = \begin{cases}
          B, & \text{if } \ Z \leq \frac{1-\lambda}{\eta} \\
          -\alpha, & \text{if } \ \frac{1-\lambda}{\eta} \leq Z \leq \frac{\lambda\left((1-\kappa)^{-1} -1 \right)+1}{\eta} \\
          0, & \text{if } \ \frac{\lambda\left((1-\kappa)^{-1} -1 \right)+1}{\eta} \leq Z.
      \end{cases}
  \end{align}\label{prop:sol_lagg_relaxation_pb}
  When $\eta > 0$.
\end{proposition}
The detailed proof can be found in the appendix.

From Lemma~1 in \cite{gao2017dynamic}, we deduce that the method of Lagrange multipliers provides the necessary and sufficient conditions for solving Equation~\eqref{PB_A}.

We denote:
\begin{align*}
  P_i = \mathbb{P}(C_i) \ \text{with} \ \begin{cases}
      C_1 &=  \{ Z \leq \frac{1-\lambda}{\eta} \}. \\
      C_2 &= \{ \frac{1-\lambda}{\eta} \leq Z \leq \frac{\lambda\left((1-\kappa)^{-1} -1 \right)+1}{\eta} \}.  \\
      C_3 &= \{ \frac{\lambda\left((1-\kappa)^{-1} -1 \right)+1}{\eta} \leq Z \}.
  \end{cases}
\end{align*}
$M^*$ can be rewritten as:
\begin{equation}
  M^* = B\ind2{C_1} - \alpha \ind2{C_2} + 0 \times \ind2{C_3}.\label{eq:sol_V}
\end{equation}

\textbf{The first constraint} \eqref{c_DCVaR_K} now takes the form $\mathbb{E} \left[ M^* + \alpha + (1 - \kappa)^{-1} (-M^* - \alpha)^+ \right] - K \leq 0$ and can be rewritten as:
\begin{equation}
    \mathbb{P}_1 \times B + \mathbb{P}_2 \times (- \alpha) + \mathbb{P}_3 \times (-\alpha (1-\kappa)^{-1}) \leq K-\alpha. \label{eq:consK_new}
\end{equation}

\textbf{The second constraint} \eqref{c_ZV_equal_V0} now takes the form $\mathbb{E} \left[ Z M^* \right] = w_0$ and can be rewritten as:

\begin{equation}
  B \mathbb{E} \left[ Z \ind2{C_1} \right] - \alpha \mathbb{E} \left[ Z \ind2{C_2} \right] = w_0. \label{eq:consx0_new}
\end{equation}

\subsection{Lagrange multipliers values and uniqueness}

In this subsection, we study the behavior of the Lagrange multipliers as functions of \(K\), where \(K\) is constrained to lie within the natural bounds \(\underline{K}\) and \(\overline{K}\) that emerge in the solution process.
\begin{align}
  &\underline{K}=\begin{dcases}
  -\alpha (1-(1-\kappa)^{-1}) \left( H_0\left(  H_1^{-1}\left(\frac{w_0}{-\alpha}\right)\right)-1\right)& \hbox{if}~~w_0<-\alpha \E[Z], \\
  (B+\alpha)H_0\left( H_1^{-1}\left( \frac{w_0+\alpha \E[Z]}{B+\alpha}\right)\right) & \hbox{if} ~~w_0\geq -\alpha \E[Z],
  \end{dcases}\label{eq:K_under_limits}\\
  &\overline{K}=B H_0\left(  H_1^{-1}(w_0/B)\right) -  \alpha\left( (1-\kappa)^{-1} \left( 1 - H_0\left(  H_1^{-1}(w_0/B)\right) \right) -1\right). \label{eq:K_bar_limits}
\end{align}

\begin{theorem}\label{th_lagrange_mul_behavior}
  The solution of \eqref{PB_A} satisfies the following results. We define for $p = \{0,1\}$: $H_p(y) = \E\left[ (Z)^p \ind2{\{Z \leq y \}} \right]$

  \begin{itemize}
    \item [(i)] If $\underline{K}< K < \overline{K}$, the solution given in \eqref{eq:sol_V} solves \eqref{PB_A} and there is a unique pair $0<\lambda<1$ and $\eta>0$ satisfying the conditions in \eqref{eq:consK_new} and \eqref{eq:consx0_new} with equality holding for \eqref{eq:consK_new}.
    
    \item [(ii)] If $K = \underline{K}$ and $w_0 = -\alpha \mathbb{E}[Z]$, then there exists a unique solution to \eqref{PB_A}. This solution corresponds to a portfolio consisting solely of the risk-free asset.
    
    \item [(iii)] If $ K = \underline{K}$ and $w_0  < -\alpha \E[Z] $, then the solution given in \eqref{eq:sol_V} solves \eqref{PB_A} with $\lambda = 1$ and $\eta = \frac{1}{(1 - \kappa) H_1^{-1} \left( \frac{\omega_0}{-\alpha} \right)}$.
    
    \item [(iv)] If $ K = \underline{K}$ and $w_0  > -\alpha \E[Z] $,  then the solution given in \eqref{eq:sol_V} solves \eqref{PB_A} with 
    $\frac{1-\lambda}{\eta} = H^{-1}\left(\frac{\omega_0 + \alpha \E[Z]}{B+\alpha}\right)$ and 
    $\frac{\lambda\left((1-\kappa)^{-1} -1 \right)+1}{\eta}=+\infty$.
    
    \item [(v)] If $ K = \overline{K}$, then the solution given in \eqref{eq:sol_V} solves \eqref{PB_A} with $\lambda=0$ and $\eta =  1/H^{-1}_1(w_0/B)$.

  \end{itemize}
\end{theorem}

The full proof is in the appendix. We begin by introducing two auxiliary variables that parametrize the Lagrange multipliers and defining a scalar function \(L\). By studying the minimum and maximum of \(L\), we identify the admissible interval \([\underline{K},\overline{K}]\) for \(K\). Since \(L\) is strictly monotonic on this interval, the equation \(L(\cdot)=K\) has exactly one solution, which in turn determines the unique pair of Lagrange multipliers.

\section{Market setting with a deterministic multi dimensional Black Scholes model}\label{sec:sse_determin}

In this section, we apply our previous results to the special case of deterministic market parameters, which yields semi-explicit solutions to equation \eqref{pb_P_DCVAR_alpha}. 

We therefore impose the following assumption on the model parameters.
\begin{asmp}\label{assumption_deterministic}
The risk free return rate $r_t$, the drift rate $\mu_t^i$, $i=1,\cdots, n$, and volatility $\sigma_t^{ij}$, $i,j=1,\cdots,n$, are all deterministic functions of $t$ for $t\in[0,T]$.
\end{asmp}

Under Assumption~\ref{assumption_deterministic}, explicit expressions for the optimal wealth and portfolio processes in ~\eqref{pb_P_DCVAR_alpha} can be derived, based on Proposition \ref{prop:sol_lagg_relaxation_pb} and Theorem~\ref{th_lagrange_mul_behavior}. Note that the definition of the state-price density process $(Z_t)_{0\le t\le T}$ in \eqref{eq:def_z} implies that, for $t\in[0,T]$, the ratio $Z_T/Z_t$ follows a log-normal distribution. In other words, the random variable $\ln\big(Z_T/Z_t\big)$ is normally distributed with mean $m(t)$ and variance $\nu^2(t)$, given by:
\begin{align*}
m(t)&  = -\int_{t}^T (r_s+\frac{1}{2}\|\theta_s\|^2 )d \tau, ~~~t\in[0,T],\label{def_m(t)}\\
\nu^2(t)& = \int_{t}^T \| \theta_s \|^2 ds, ~~~t\in[0,T].
\end{align*}
As a special case, when $t = 0$, the logarithm $\ln(Z_T)$ is normally distributed with mean $m(0)$ and variance $\nu^2(0)$. Furthermore, we can directly compute the expectation $\mathbb{E}[Z_T] = e^{-\int_{0}^T r_s\,ds}$.\\

\begin{lemma}
  Under Assumption \ref{assumption_deterministic}, we recall that $\ln(Z_T)$ follows the normal distribution with mean and variance being $m(0)$ and $\nu^2(0)$. So, we have $H_0(y) = \Phi(F_1(y))$ with $F_1(y) = (ln(y)-m(0))/\nu(0)$ and $H_1(y) = e^{-\int_{0}^T r_sds}\Phi(F_2(y))$ with $F_2(y) = (ln(y)-m(0))/\nu(0) - \nu(0)$. \label{lem:H0H1}\\
\end{lemma}

\begin{proposition}
  Under Assumption \ref{assumption_deterministic}, the bounds given in Theorem \ref{th_lagrange_mul_behavior} verify,
  \begin{align*}
    \underline{K}=&\begin{dcases}
        -\alpha (1-(1-\kappa)^{-1}) \left( \Phi\left(  \Phi^{-1}\left(\frac{w_0}{-\alpha e^{-\int_{0}^T r_sds}}\right) +\nu(0) \right)-1\right)& \hbox{if}~~w_0<-\alpha e^{-\int_{0}^T r_sds}, \\
        (B+\alpha)\Phi\left(  \Phi^{-1}\left( \frac{1}{B+\alpha}\left(\frac{w_0}{e^{-\int_{0}^T r_sds}}+\alpha\right)\right)+\nu(0)\right) & \hbox{if} ~~w_0\geq -\alpha e^{-\int_{0}^T r_sds},
    \end{dcases}\\
    \overline{K}=& B \Phi\left( \Phi^{-1}\left(\frac{w_0}{B e^{-\int_{0}^T r_sds}}\right) +\nu(0)\right) \notag \\ 
    &-  \alpha\left( (1-\kappa)^{-1} \left( 1 - \Phi\left( \Phi^{-1}\left(\frac{w_0}{B e^{-\int_{0}^T r_sds}}\right) +\nu(0)\right) \right) -1\right). \label{eq:K_bar_assumption_1}
  \end{align*} \label{prop:K_deterministic}
\end{proposition}

\begin{proof}
  let us start with the following observation from the corollary of Lemma \ref{lem:H0H1},
  \begin{equation*}
    F_1(F_2^{-1}(y)) = y + \nu(0),
  \end{equation*}
  then, we can derive a formula for $H_0(H_1^{-1}(.))$, 
  \begin{align*}
    H_0(H_1^{-1}(y)) &= \Phi\left(F_1\left(F_2^{-1}\left(\Phi^{-1}\left(\frac{y}{e^{-\int_{0}^T r_sds}} \right)\right)\right)\right),\\
    &= \Phi\left( \Phi^{-1}\left( \frac{y}{e^{-\int_{0}^T r_sds}}\right) + \nu(0) \right).
  \end{align*}
  Replace $H_0(H_1^{-1}(.))$ in \eqref{eq:K_under_limits} and \eqref{eq:K_bar_limits} completes the proof.
\end{proof}

\begin{theorem}\label{thm_solution}
  Given Assumption \ref{assumption_deterministic}, the optimal solution to Equation \eqref{PB_A} is as follows, for all $t\in[0,T]$.\\
  (i) If $\underline{K}<K <\overline{K}$, the optimal wealth process is
  \begin{align}
  V^*_t&=e^{m(t)+\frac{\nu^2(t)}{2}}\bigg((B+\alpha)\Phi\big( Y_1(t)-\nu(t)\big)-\alpha\Phi\big(Y_2(t)-\nu(t)\big) \bigg),\label{thm_2_V_carac}
  \end{align}
  and the optimal portfolio policy is
  \begin{align}
  \beta^*_t \odot S_t &= \frac{e^{m(t)+\frac{\nu^2(t)}{2}}}{\sqrt{2\pi}\nu(t)} \bigg((B+\alpha)e^{-\frac{(Y_1(t)-\nu(t))^2}{2}} -\alpha e^{-\frac{(Y_2(t)-\nu(t))^2}{2}}\bigg)
  (\sigma_t\sigma_t^{\top})^{-1}b_t, \label{thm_2_pi_carac}
  \end{align}
  with the optimal objective value being
  \begin{align}
  E[V_T^*]=&(B+\alpha)\Phi\big(Y_1(0)\big)-\alpha\Phi\big(Y_2(0)\big), \label{thm_lpm_dq<1_obj}
  \end{align}
  where $Y_1(t)$ and $Y_2(t)$ are defined as
  \begin{align*}
  Y_1(t)=\frac{\ln\left(\frac{1-\lambda}{\eta Z_t}\right)-m(t)}{\nu(t)},
  ~Y_2(t)= \frac{\ln\left(\frac{\lambda\left((1-\kappa)^{-1}-1\right)+1}{\eta Z_t}\right)-m(t)}{\nu(t)},
  \end{align*}
  with $0<\lambda<1$ and $0<\eta$ being the solution of the following two equations,
  \begin{align}
  &(B+\alpha)\Phi\big(Y_1(0)\big)-\alpha\left(\left(1-\kappa\right)^{-1}-1\right)\left(1-\Phi\left(Y_2(0)\right)\right)=K,
  \label{thm_lpm_dq<1_eq1}\\
  &(B+\alpha)\Phi\big(Y_1(0)-\nu(0)\big)-\alpha\Phi\big(Y_2(0)-\nu(0)\big)
  =e^{\int_0^T r_sds }w_0.
  \label{thm_lpm_dq<1_eq2}
  \end{align}
  
  (ii) If $K = \underline{K}$ and $w_0 = -\alpha e^{-\int_0^T r_sds }$, then $K=0$, and the optimal wealth process and portfolio policy correspond to a portfolio consisting solely of the risk-free asset. Consequently, we have $\E(V_T)=-\alpha$.

  (iii) If $K = \underline{K}$ and $\omega_0 < -\alpha e^{-\int_{0}^T r_sds}$, the optimal wealth process is
  \begin{align}
  V^*_t&=e^{m(t)+\frac{\nu^2(t)}{2}}\bigg(-\alpha\Phi\big(Y_2(t)-\nu(t)\big) \bigg),\label{thm_2_V_carac_iii}
  \end{align}
  and the optimal portfolio policy is
  \begin{align}
    \beta^*_t \odot S_t &=\frac{e^{m(t)+\frac{\nu^2(t)}{2}}}{\sqrt{2\pi}\nu(t)} \bigg(-\alpha e^{-\frac{(Y_2(t)-\nu(t))^2}{2}}\bigg)
  (\sigma_t\sigma_t^{\top})^{-1}b_t, \label{thm_2_pi_carac_iii}
  \end{align}
  with the optimal objective value being
  \begin{align}
  E[V_T^*]=& -\alpha\Phi\big(Y_2(0)\big), \label{thm_lpm_dq<1_obj_iii}
  \end{align}
  where $Y_2(t)$ is define as in (i) where $\lambda=1$ and $\eta=1/{((1-\kappa)\rho)}$ with $\rho =F_2^{-1}\left(\Phi^{-1}\left( \frac{\omega_0 e^{\int_0^T r_sds }}{-\alpha }\right)\right) =$ and $\delta = 0$.

  (iv) If $K = \underline{K}$ and $\omega_0 > -\alpha e^{-\int_{0}^T r_sds}$, the optimal wealth process is
  \begin{align}
  V^*_t&=e^{m(t)+\frac{\nu^2(t)}{2}}\bigg((B+\alpha)\Phi\big( Y_1(t)-\nu(t)\big)-\alpha \bigg),\label{thm_2_V_carac_iv}
  \end{align}
  and the optimal portfolio policy is
  \begin{align}
    \beta^*_t \odot S_t &= \frac{e^{m(t)+\frac{\nu^2(t)}{2}}}{\sqrt{2\pi}\nu(t)} \bigg((B+\alpha)e^{-\frac{(Y_1(t)-\nu(t))^2}{2}} \bigg)
  (\sigma_t\sigma_t^{\top})^{-1}b_t, \label{thm_2_pi_carac_iv}
  \end{align}
  with the optimal objective value being
  \begin{align}
  E[V_T^*]=&(B+\alpha)\Phi\big(Y_1(0)\big) -\alpha, \label{thm_lpm_dq<1_obj_iv}
  \end{align}
  where $Y_2(t)$ is define as in (i) where $\rho = +\infty$ and $\delta = \frac{1-\lambda}{\eta}= F_2^{-1}\left(\Phi^{-1}\left( \frac{\omega_0 e^{\int_0^T r_sds }+ \alpha}{B + \alpha}\right)\right)$ with $\lambda=1$ and $\eta=0$.

  (v) If $K = \overline{K}$, the optimal wealth process is
  \begin{align}
  V^*_t&=e^{m(t)+\frac{\nu^2(t)}{2}}\bigg((B+\alpha)\Phi\big( Y_1(t)-\nu(t)\big) \bigg),\label{thm_2_V_carac_iv}
  \end{align}
  and the optimal portfolio policy is
  \begin{align}
    \beta^*_t \odot S_t &=\frac{e^{m(t)+\frac{\nu^2(t)}{2}}}{\sqrt{2\pi}\nu(t)} \bigg((B+\alpha)e^{-\frac{(Y_1(t)-\nu(t))^2}{2}} \bigg)
  (\sigma_t\sigma_t^{\top})^{-1}b_t, \label{thm_2_pi_carac_v}
  \end{align}
  with the optimal objective value being
  \begin{align}
  E[V_T^*]=&B\Phi\big(Y_1(0)\big),\label{thm_lpm_dq<1_obj_v}
  \end{align}
  where $Y_1(t)$ is define as in (i) where $\lambda=0$ and $\eta=1/\bar{\delta}$ with $\bar{\delta} = F_2^{-1}\left(\Phi^{-1}\left(w_0/\left(B e^{-\int_0^T r_sds }\right)\right)\right)$.
  
  \end{theorem}

\begin{proof}
  (i) We first consider the case where $\underline{K} < K < \overline{K}$. The following result holds for any parameter $c > 0$. Note that $Z_t$ is $\mathcal{F}_t$-adapted, and from Theorem \ref{th_lagrange_mul_behavior}, we have $\lambda > 0$ and $\eta > 0$.
  \begin{align}
    &\E\left[\frac{Z_T}{Z_t}\1_{Z_T\leq c}  ~\Big| ~\cF_t\right]\notag =\E\left[e^{\ln \frac{Z_T}{Z_t}}\1_{\ln \frac{Z_T}{Z_t}\leq \ln \frac{c}{ Z_t} } ~\Big|~\cF_t\right]\notag =e^{m(t) + \frac{\nu^2(t)}{2}}\Phi\left( \frac{\ln\frac{c}{Z_t}  -m(t)}{\nu(t)} -\nu(t)\right),\notag \\ \label{thm_2_proof_ex1}
    \end{align}
  where the last equality is based on Lemma \ref{lem_truncated_expectation} presented in appendix. The discounted optimal wealth process is a martingale under the probability measure $\tilde{\mP}$ (see \cite{karatzas1998methods}), meaning that we have:
  \begin{align*}
  V^*_t=\E\left[\frac{Z_T}{Z_t}V_T^*~\Big|~\cF_t\right], ~t\in[0,T].
  \end{align*}
  Using the expression of $V_T^*$ from \eqref{eq:sol_V}, we can derive $V^*_t$ as follows:
  \begin{align}
  V^*_t&=\E\Big[\frac{Z_T}{Z_t}\left(\left(B+\alpha\right)\1_{\eta Z_T-1\leq \lambda} -\alpha\1_{\frac{\eta Z_T-1}{(1-\kappa)^{-1}-1}\leq \lambda}\right) ~|~\cF_t\Big] , ~t\in[0,T]. \label{thm_2_proof_ex2}
  \end{align}
  Applying \eqref{thm_2_proof_ex1} to \eqref{thm_2_proof_ex2} gives rise to the result in \eqref{thm_2_V_carac}.

  Under Assumption \ref{assumption_deterministic}, we can assume $V^*_t$ as a deterministic function of $Z_t$ and $t$, i.e., there exists a function $G(\cdot,\cdot)$ such that $V^*_t = G(Z_t,t)$. Now, let us determine the functional form of $G(Z_t,t)$. 

  Applying Itô’s Lemma yields:
  \begin{align}
  dG(Z_t,t)&=\frac{\pp G(Z_t,t)}{\pp Z_t}dZ_t +\frac{\pp G(Z_t,t)}{\pp t}dt +\frac{1}{2} \frac{\pp^2 G(Z_t,t)}{\pp Z_t^2}(\theta_t^2Z_t^2)dt\notag \\
  &=(-Z_t\frac{\pp G(Z_t,t)}{\pp Z_t}r_t+\frac{\pp G(Z_t,t)}{\pp t}+\frac{1}{2}\frac{\pp^2 G(Z_t,t)}{\pp Z_t^2}Z_t^2\|\theta_t\|^2 )dt\notag\\
  &~~-\frac{\pp G(Z_t,t)}{\pp Z_t}Z_t\theta_t^{\top} dW_t.\label{thm_2_proof_G}
  \end{align}
  Comparing the diffusion term in \eqref{thm_2_proof_G} with that of the wealth process in \eqref{def_wealth} leads to, for $t\in[0,T]$:
  \begin{align}
    (\beta^*_t \odot S_t)^\top \sigma_t =-\frac{\pp G(Z_t,t)}{\pp Z_t} Z_t \theta_t^{\top}. \label{thm_2_proof_pi_carac}
  \end{align}
  By applying the definition of $\theta$, right‑multiplying both sides of \eqref{thm_2_proof_pi_carac} by $\sigma_t^{-1}$ and then taking the transpose yields:

  \begin{align*}
    \beta^*_t \odot S_t=-\frac{\pp G(Z_t,t)}{\pp Z_t}Z_t(\sigma_t\sigma_t^{\top})^{-1}b_t.
  \end{align*}
  Thus, differentiating \eqref{thm_2_V_carac} with respect to $Z_t$ further leads to the result in \eqref{thm_2_pi_carac}.

  From Proposition \ref{prop:sol_lagg_relaxation_pb} and Theorem \ref{th_lagrange_mul_behavior}, we can identify Lagrange multipliers $\eta$ and $\lambda$ by solving equations \eqref{c_DCVaR_K} and \eqref{c_ZV_equal_V0}. Equation \eqref{c_ZV_equal_V0} can be written explicitly  by letting $t=0$ in \eqref{thm_2_V_carac} which yields \eqref{thm_lpm_dq<1_eq2}. Based on \eqref{eq:consK_new}, we have
  \begin{align*}
  K &=(B+\alpha) H_0\left(\frac{1-\lambda}{\eta}\right) - \alpha \left( (1-\kappa)^{-1} -1 \right) \left( 1 - H_0\left(\frac{\lambda((1-\kappa)^{-1}-1)+1}{\eta}\right) \right), \notag \\
  & = (B+\alpha) \Phi\left(F_1\left(\frac{1-\lambda}{\eta}\right)\right) - \alpha \left( (1-\kappa)^{-1} -1 \right) \left( 1 - \Phi\left(F_1\left(\frac{\lambda((1-\kappa)^{-1}-1)+1}{\eta}\right) \right)\right),  \notag \\
  &= (B+\alpha)\Phi\big(Y_1(0)\big)-\alpha\left(\left(1-\kappa\right)^{-1}-1\right)\left(1-\Phi\left(Y_2(0)\right)\right). \notag \\
  \end{align*}
  We apply the same techniques as in \eqref{eq:sol_V} to derive \eqref{thm_lpm_dq<1_obj}.

  Result (ii) follows directly from Theorem \ref{th_lagrange_mul_behavior}. For result (iii), (iv) and (v), the reasoning is the same as in case (i).
\end{proof}

\subsection{Asymptotic behavior when $B$ is large}
\medskip
\begin{proposition}
  Under Assumption \ref{assumption_deterministic} we have:
  \begin{align*}
    & \lim_{B \rightarrow +\infty} \underline{K}=\begin{dcases}
    -\alpha (1-(1-\kappa)^{-1}) \left( \Phi\left(  \Phi^{-1}\left(\frac{w_0}{-\alpha e^{-\int_{0}^T r_sds}}\right) +\nu(0) \right)-1\right)& \hbox{if}~~w_0<-\alpha e^{-\int_{0}^T r_sds},\\
    \frac{w_0}{e^{-\int_{0}^T r_sds}} + \alpha & \hbox{if} ~~w_0\geq -\alpha e^{-\int_{0}^T r_sds},
    \end{dcases}\\
    &\lim_{B \rightarrow +\infty} \overline{K}= \frac{w_0}{e^{-\int_{0}^T r_sds}} - \alpha\left((1-\kappa)^{-1}-1 \right).
  \end{align*}
\end{proposition}
\begin{proof}
  The result follows directly from the Propostion $\ref{prop:K_deterministic}$. For $\underline{K}$ when $w_0 \geq -\alpha e^{-\int_{0}^T r_sds}$ we have,\\
  \begin{align*}
    \underline{K} &= (B+\alpha)\Phi\left(  \Phi^{-1}\left( \frac{1}{B+\alpha}\left(\frac{w_0}{e^{-\int_{0}^T r_sds}}+\alpha\right)\right)+\nu(0)\right),\\
    &\underset{B \sim +\infty}{\sim}(B+\alpha)\Phi\left(  \Phi^{-1}\left( \frac{1}{B+\alpha}\left(\frac{w_0}{e^{-\int_{0}^T r_sds}}+\alpha\right)\right)\right),\\
    &\underset{B \sim +\infty}{\sim}\frac{w_0}{e^{-\int_{0}^T r_sds}}+\alpha.
  \end{align*}
  For $\overline{K}$, in a similar way we have $H_1^{-1}(w_0/B) \underset{B \sim +\infty}{\sim} \frac{w_0}{B \times e^{-\int_{0}^T r_sds}}$, which gives us the result.
\end{proof}

\begin{rem}
For now, $M^*$ does not represent the final value of the most profitable portfolio satisfying 
$V_0 = w_0$, $\alpha = VaR_\kappa(-V_T)$, and $DCVaR_\kappa(-V_T) = K$. 
Indeed, without minimizing with respect to $\alpha$, these equalities do not hold.  

The bounds $\underline{K}$ and $\overline{K}$ are therefore only technical values used in the resolution 
and cannot be interpreted as the $DCVaR$ of a portfolio. They still provide intuition: 
for instance, if $-\alpha$ is of the same order as $w_0$ and $\kappa = 0.99$, then $\overline{K}$ 
is essentially driven by $-\alpha$ and is about one hundred times larger.  

When $w_0 > -\alpha e^{-\int_{0}^T r_s\,ds}$, the lower bound 
\[
\underline{K} = \frac{w_0}{e^{-\int_{0}^T r_s\,ds}} + \alpha
\]
represents the discounted distance between $w_0$ and $-\alpha e^{-\int_{0}^T r_s\,ds}$.
\end{rem}\label{rem:remark_limit_cases}

To conclude this subsection, it is necessary in practice to choose a sufficiently large value of $B$ to approximate the true result. However, we cannot let $B$ tend to infinity, in order to preserve the validity of the results in Theorem~\ref{th_lagrange_mul_behavior}, and thus also those of Theorem~\ref{thm_solution}.

\subsection{Procedure to find the optimal $\alpha$}
Up to this point, $\alpha$ has been treated as a fixed parameter, and all results derived in Theorem~\ref{thm_solution} are accordingly $\alpha$-dependent. In particular, this includes the expected return of the portfolio solution to Equation ~\eqref{pb_P_DCVAR_alpha}, denoted by $R(\alpha)$. We extend the definition of $R(\alpha)$ by assigning it the value $+\infty$ whenever the optimization admits no solution. Then, we have:
\begin{align}
  R(\alpha) = \begin{dcases}
    (B+\alpha) \Phi\big(Y_1(0,\alpha)\big) - \alpha \Phi\big(Y_2(0,\alpha)\big), & \text{if}~ \underline{K}(\alpha) < K < \overline{K}(\alpha),\\
    - \alpha \Phi\big(Y_2(0,\alpha)\big), & \text{if}~ K= \underline{K}(\alpha), \ \text{and} \ \omega_0 < -\alpha e^{-\int_{0}^T r_sds},\\
    (B+\alpha) \Phi\big(Y_1(0,\alpha)\big)-\alpha, & \text{if}~ K= \underline{K}(\alpha), \ \text{and} \ \omega_0 > -\alpha e^{-\int_{0}^T r_sds},\\
    (B+\alpha)\Phi\big(Y_1(0,\alpha)\big), & \text{if}~ K= \overline{K}(\alpha),\\
    +\infty & \text{otherwise}.
  \end{dcases} \label{def_J}
\end{align}
Here, the functions $Y_1$ and $Y_2$ depend on $\alpha$ through $\lambda(\alpha)$ and $\eta(\alpha)$, with $\delta$ and $\rho$ taking specific values in each boundary case.
We exclude here the case $K=0$.

With the Lemma \ref{lem_result_is_convex}, since the function $\beta \mapsto -\mathbb{E}[V_T]$ and the function 
\begin{align*}
  &(\alpha, \beta) \mapsto \mathbb{E} \left[V_T + \alpha + (1 - \kappa)^{-1} \max(-V_T-\alpha, 0) \right]
\end{align*}
are convex, it follows that $\alpha \mapsto -R(\alpha)$ is also convex.

Finally, since the objective function $-R(\alpha)$ is convex with respect to $\alpha$, we can apply a gradient-based search procedure to determine the optimal value $\alpha^* := \arg \min_{\alpha} -R(\alpha)$.\\

\begin{alg}[Gradient Descent]
  The complete procedure is as follows: First, we fix a large value for $B$, which remains constant throughout the optimization. We assume that for the initial choice of $\alpha$, we select $K \in [\underline{K}(\alpha), \overline{K}(\alpha)]$.

  \begin{itemize}
      \item \underline{Step 0}: Choose an initial value $\alpha \leftarrow \alpha_0$ and a small positive number $\epsilon > 0$ as the stopping criterion. Proceed to Step 1.
      
      \item \underline{Step 1}: For a given $\alpha$, set $\hat{\alpha} \leftarrow \alpha + \zeta$, then compute $R(\alpha)$ and $R(\hat{\alpha})$ using (\ref{def_J}). Computing $R(\alpha)$ requires determining the unique $\delta$ such that the monotonically increasing function $L$ from \eqref{eq:defL} equals $K$. A unique $\rho^*$ can then be determined by substituting $\delta^*$ into \eqref{eq:H1_carac}. Finally, the pair $(\lambda, \eta)$ can be determined through a one-to-one correspondence, which in turn allows us to compute $R(\alpha)$.

      \item \indent \underline{Step 2} Compute $\big( R(\hat{\alpha})-R(\alpha) \big)/\zeta$. If it falls below \(\epsilon\), return $\alpha$ as the optimal solution. Otherwise, let $\alpha=\alpha+ \vartheta\cdot \big( R(\hat{\alpha})-R(\alpha) \big)/\zeta$. Go to Setp 1.
      
  \end{itemize}

  For each iteration of Step 1, it is only required to determine the antecedent of $K$ through the function $L$, which is monotonic. This process is not expected to be computationally intensive. In practice, using the default tolerances \(\epsilon\) and \(\vartheta\) of the Python \texttt{scipy.optimize.minimize} function, we executed the algorithm 2000 times in under one minute.

\end{alg}

\section{Numerical Experiments} \label{sec:numerics}
In this section, we apply the previously described method to conduct numerical experiments. We study a simple case in order to observe the resulting solution and analyze the corresponding investment strategy over given market paths. Then, we examine how the expected value of the portfolio evolves with respect to the capital-at-risk parameter $K$.

\subsection{A simple example}
We consider a financial market consisting of four risky assets and a risk-free asset (cash). The dynamics of the risky assets follow a deterministic, multi-dimensional Black--Scholes model in continuous time, over a finite investment horizon of one year. Each risky asset \( (S_t^i)_{0 \le t \le T} \), for \( i = 1, \dots, 4 \), evolves according to Equation~\eqref{eq:asset_dynamics}, with drift coefficient \( \mu_i \) and volatility structure determined by the matrix \( \sigma \) defined below.

The model parameters used in this study are as follows. The investment horizon is set to \( T = 1.0 \) year and the constant risk-free interest rate is \( r = 0.02 \). All risky assets share the same initial price \( S_0 = 100 \). The expected returns are given by the vector \( \mu = [0.09, 0.15, 0.21, 0.12] \), and the volatilities by \( \mathbf{vol} = [0.08, 0.12, 0.15, 0.08] \). The correlation matrix is defined as:
\[
\rho = \begin{pmatrix}
1 & 0.2 & -0.3 & 0 \\
0.2 & 1 & 0.15 & -0.2 \\
-0.3 & 0.15 & 1 & 0.3 \\
0 & -0.2 & 0.3 & 1 \\
\end{pmatrix}
\]

The initial portfolio value is set to $w_0 = 100$, the capital-at-risk constraint is $K = 30$, the quantile level is $\kappa = 0.99$, and we set $B = 500$.

Let \( \operatorname{vol}_i \) denote the volatility of asset \( i \), for \( i = 1, \dots, n \), and define the volatility vector \( \mathbf{vol} = (\operatorname{vol}_1, \dots, \operatorname{vol}_n)^\top \).  
Given the correlation matrix \( \rho \in \mathbb{R}^{n \times n} \), the covariance matrix is constructed as
\[
\Gamma = D \rho D, \quad \text{where } D = \operatorname{diag}(\mathbf{vol}).
\]
Then we compute the matrix \( \sigma \in \mathbb{R}^{n \times n} \) such that \( \Gamma = \sigma \sigma^\top \), typically via Cholesky decomposition.

This setting enables the generation of tractable market scenarios, making it suitable for illustrating and evaluating the proposed optimization framework.

Under this setup, the optimal expected terminal wealth is $\mathbb{E}[V^*_T] = 141.78$, obtained using the auxiliary parameter $\alpha^* = -121.14$. The following figures depict different market paths, showing asset prices $S_t$, optimal wealth $V^*_t$, asset holdings $\beta_t$, and position values $\beta^*_t \odot S_t$.

\begin{figure}[H]
  \centering

  \begin{minipage}[b]{0.49\textwidth}
      \centering
      \includegraphics[width=\linewidth]{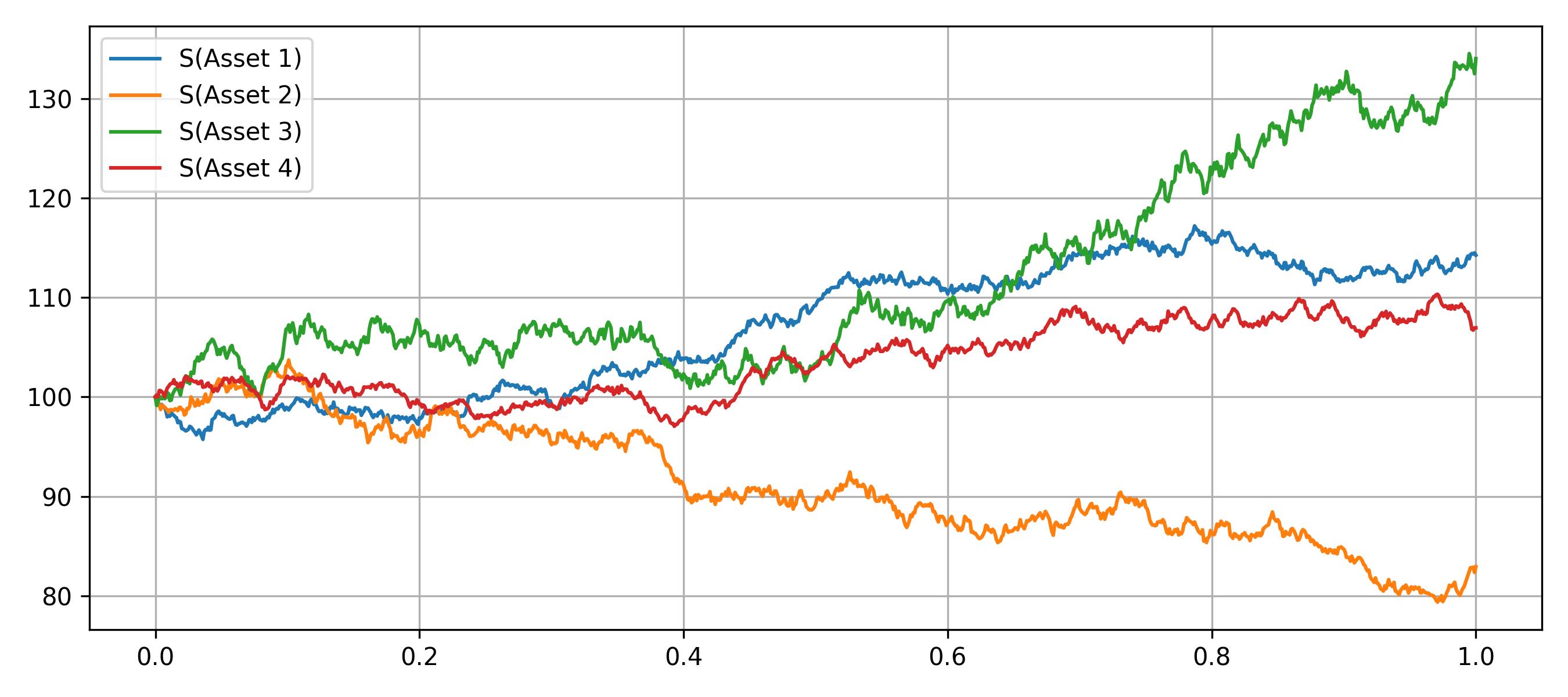}
      \caption*{Evolution of asset prices $S_t$.}
  \end{minipage}
  \hfill
  \begin{minipage}[b]{0.49\textwidth}
      \centering
      \includegraphics[width=\linewidth]{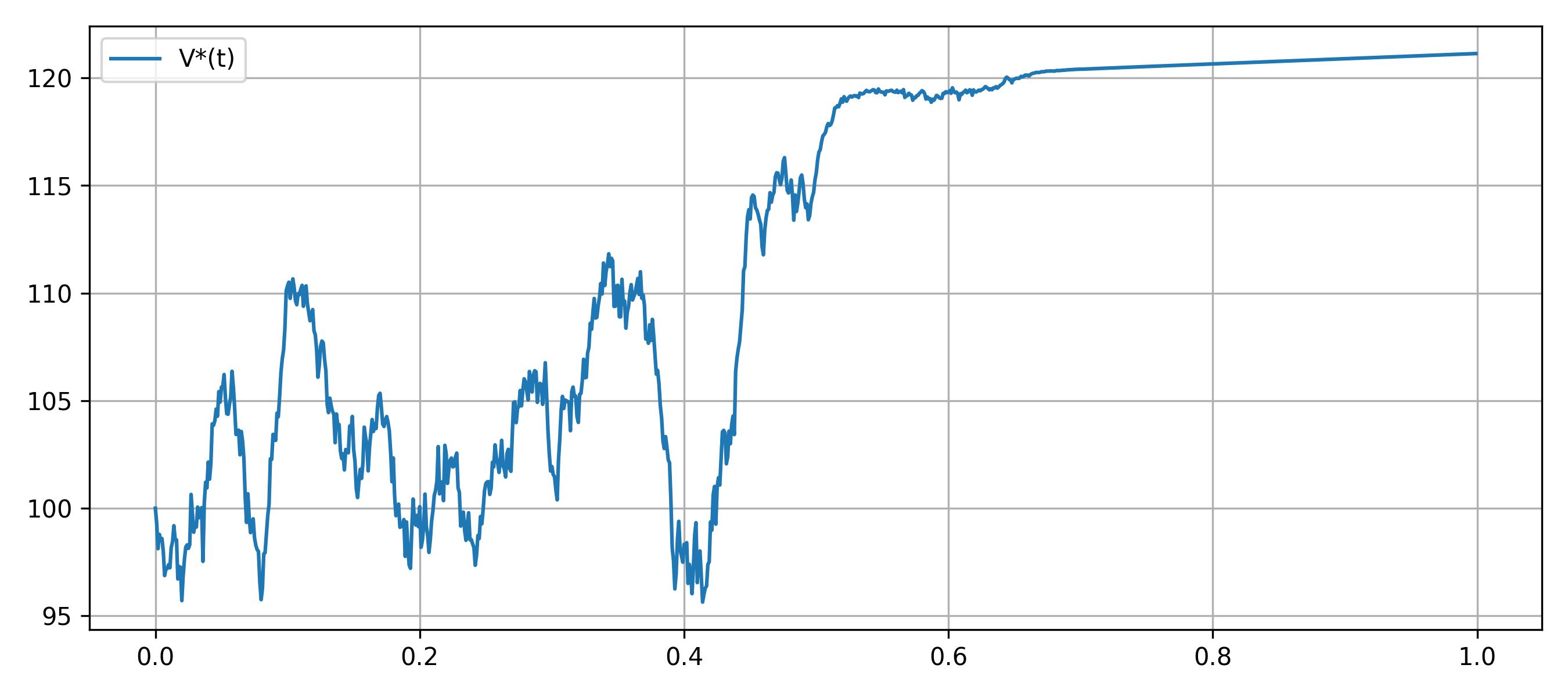}
      \caption*{Evolution of wealth $V^*_t$.}
  \end{minipage}

  \vspace{1em}

  \begin{minipage}[b]{0.72\textwidth} 
      \centering
      \includegraphics[width=\linewidth]{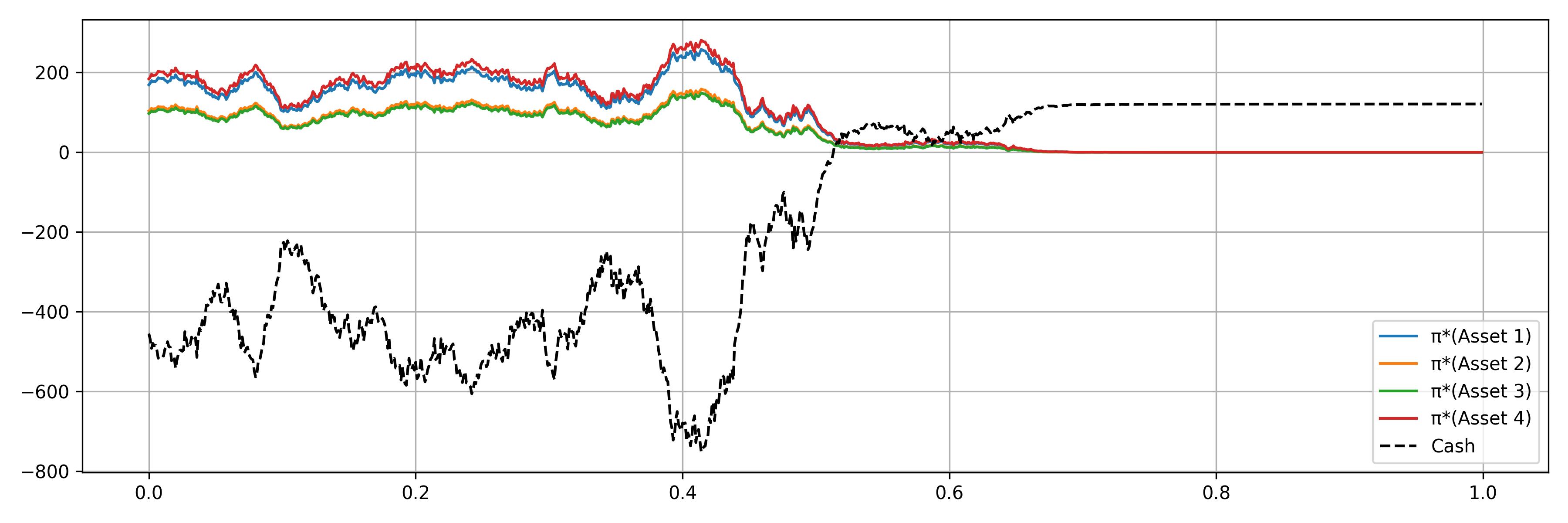}
      \caption*{Position value $\beta_t \odot S_t(=\pi_t^*)$ and cash position.}
  \end{minipage}

  \caption{Portfolio dynamics for seed 2 with $V^*_T = 121.14$}
  \label{fig:combined_K30_seed2}
\end{figure}

\begin{figure}[H]
  \centering

  \begin{minipage}[b]{0.49\textwidth}
      \centering
      \includegraphics[width=\linewidth]{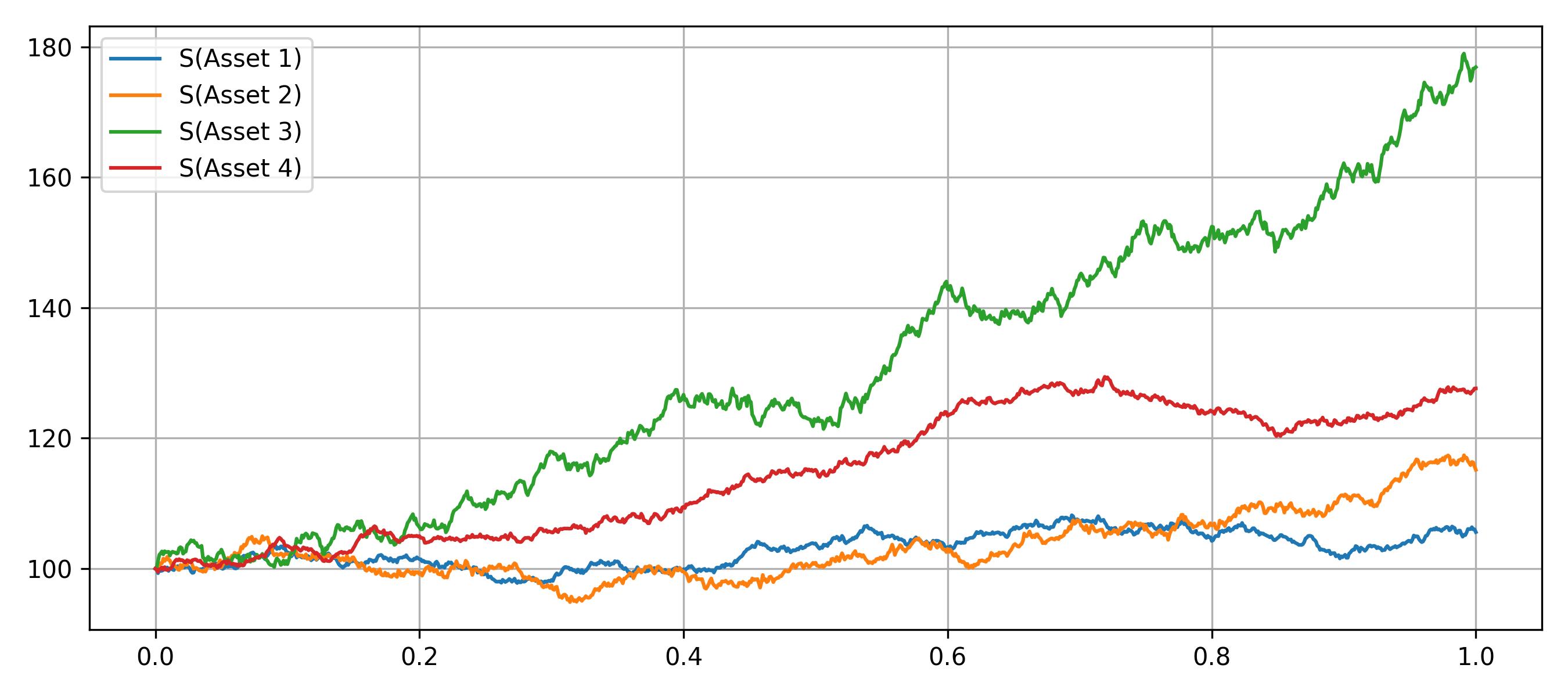}
      \caption*{Evolution of asset prices $S_t$.}
  \end{minipage}
  \hfill
  \begin{minipage}[b]{0.49\textwidth}
      \centering
      \includegraphics[width=\linewidth]{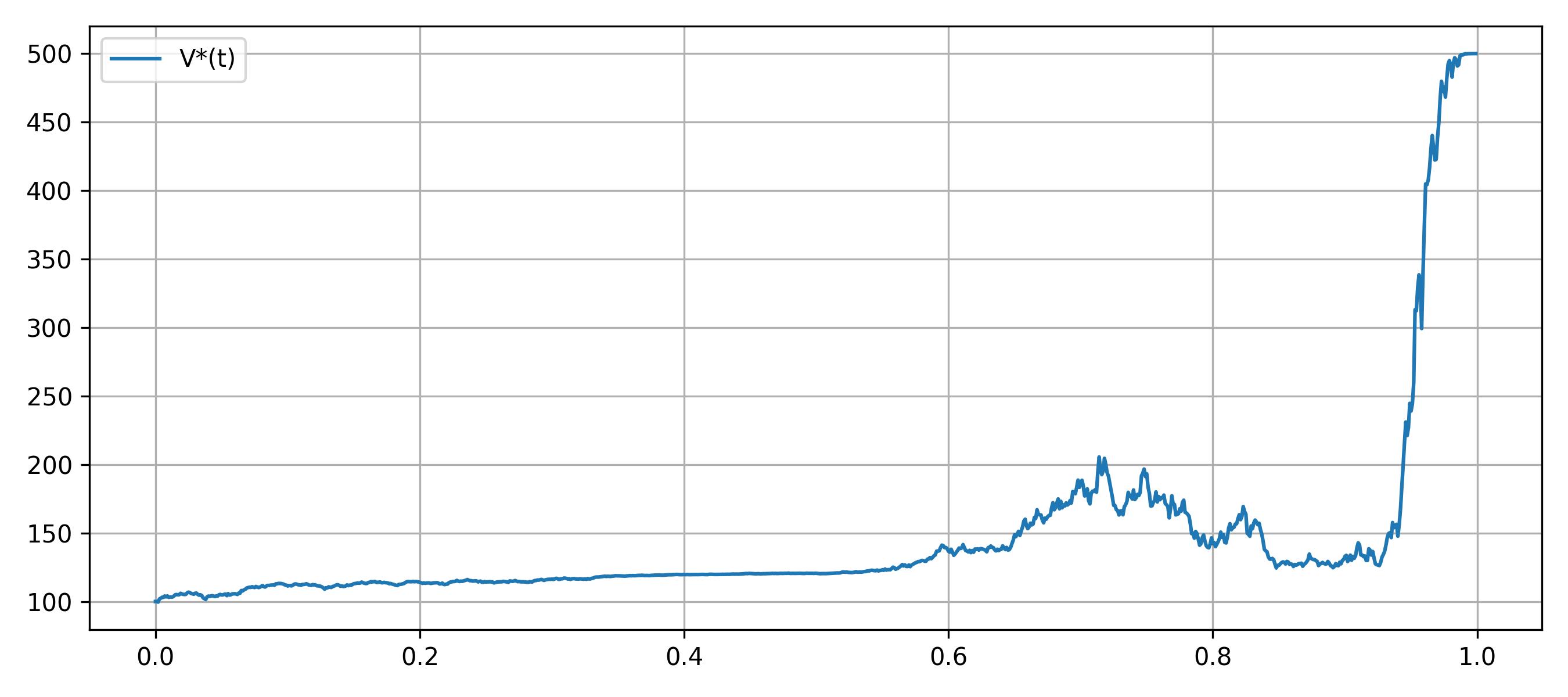}
      \caption*{Evolution of wealth $V^*_t$.}
  \end{minipage}

  \vspace{1em}

  \begin{minipage}[b]{0.72\textwidth}
      \centering
      \includegraphics[width=\linewidth]{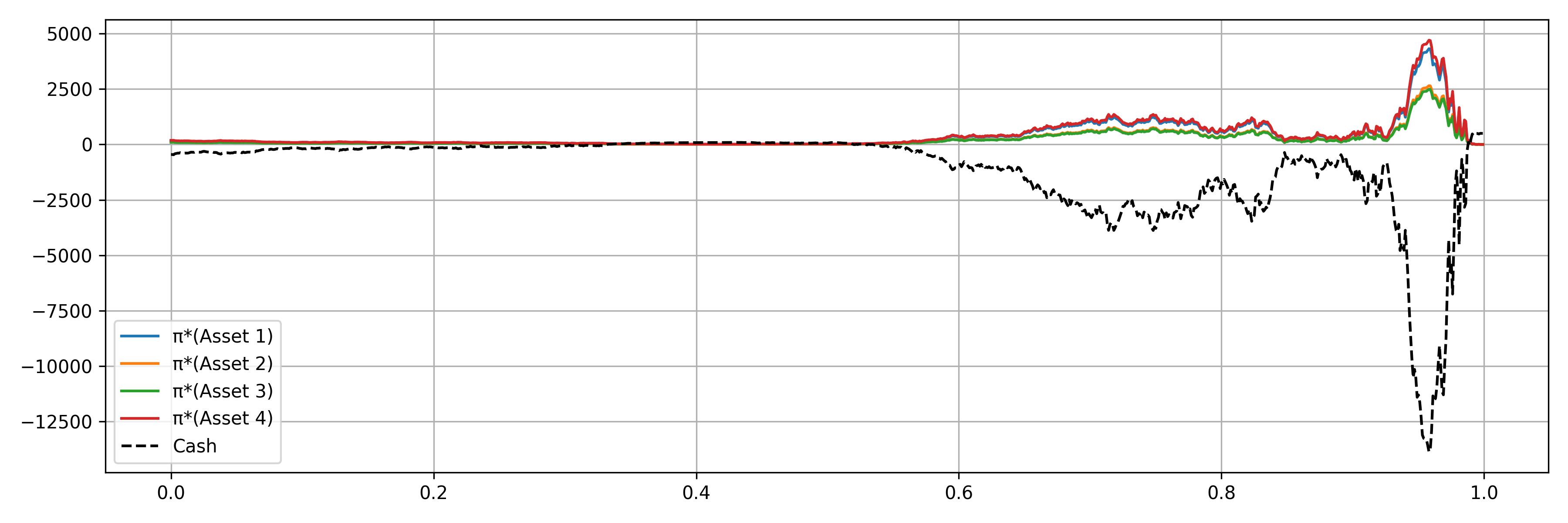}
      \caption*{Position value $\beta_t \odot S_t(=\pi_t^*)$ and cash position.}
  \end{minipage}

  \caption{Portfolio dynamics for seed 6 with $V^*_T = 500$}
  \label{fig:combined_K30_seed66}
\end{figure}

\begin{figure}[H]
  \centering

  \begin{minipage}[b]{0.49\textwidth}
      \centering
      \includegraphics[width=\linewidth]{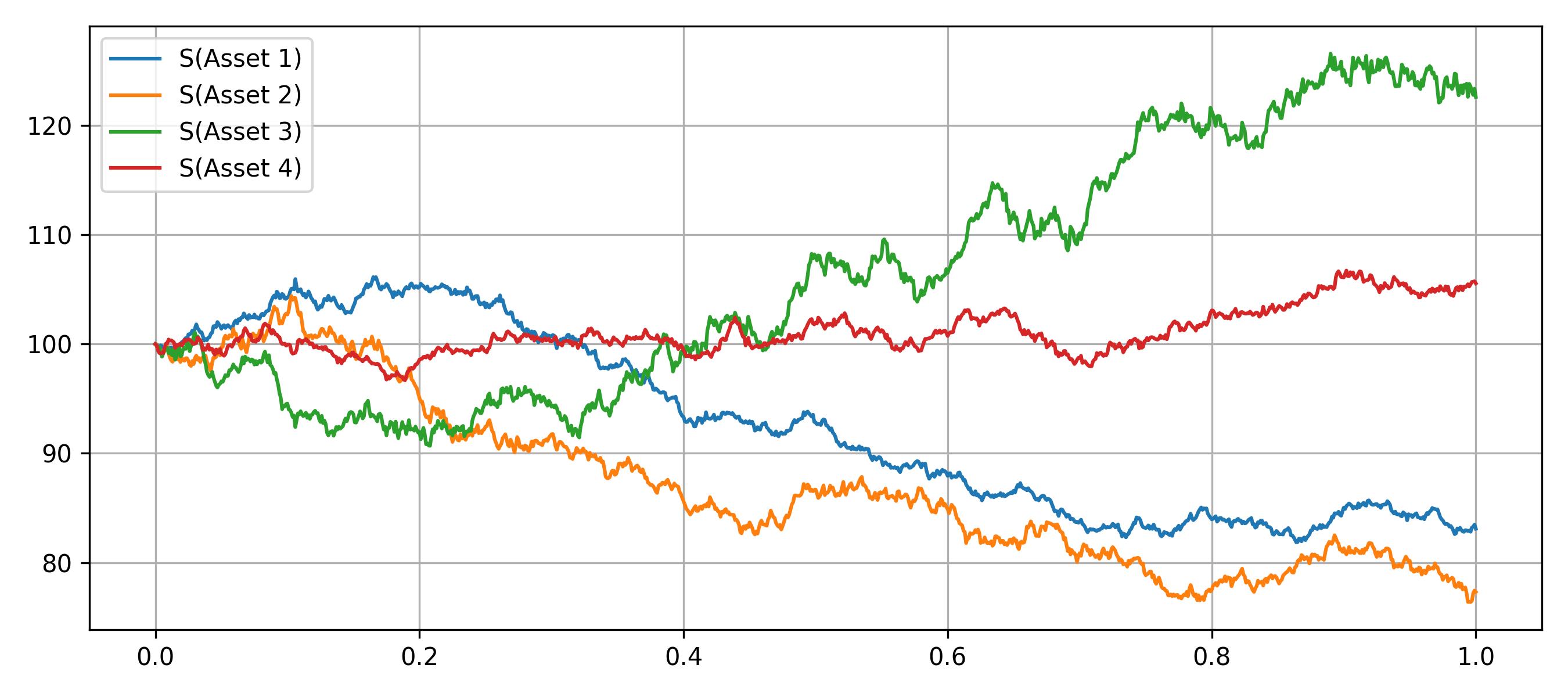}
      \caption*{Evolution of asset prices $S_t$.}
  \end{minipage}
  \hfill
  \begin{minipage}[b]{0.49\textwidth}
      \centering
      \includegraphics[width=\linewidth]{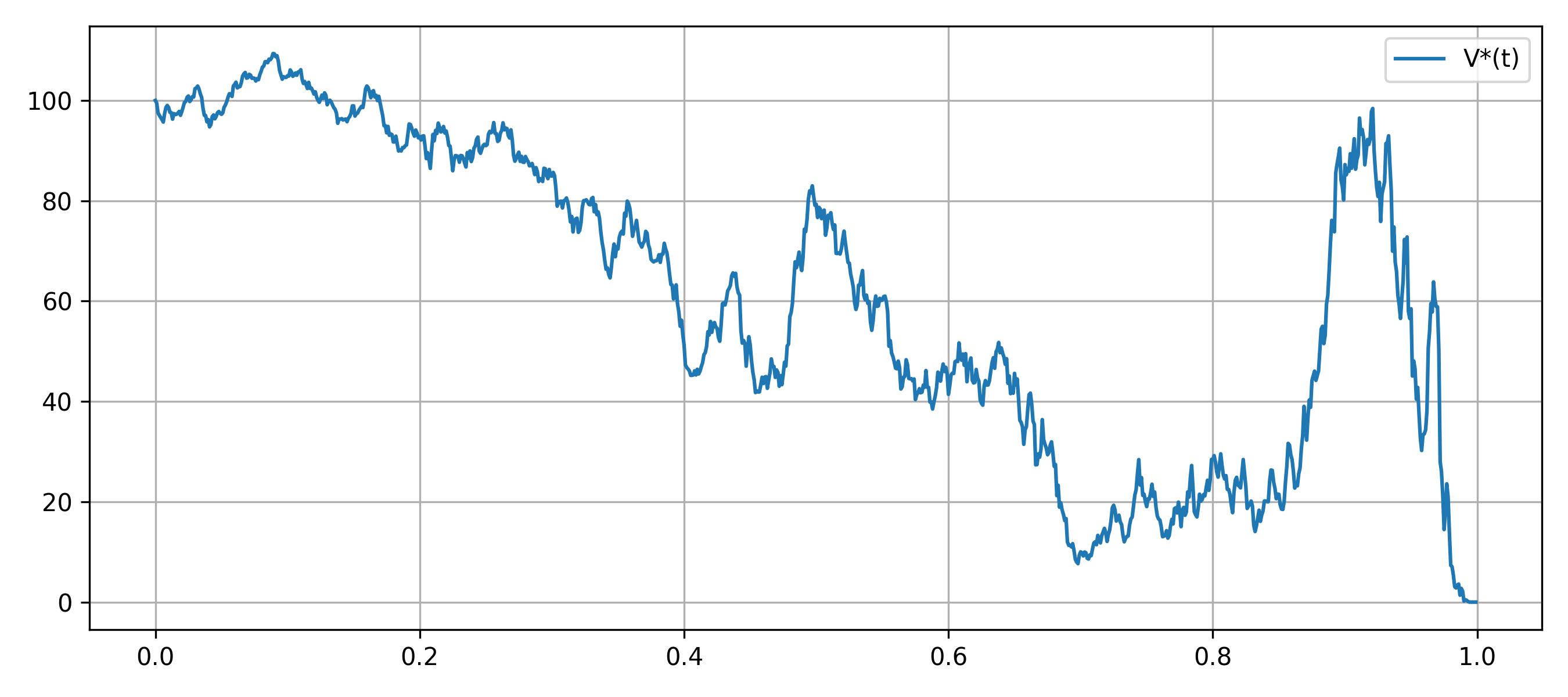}
      \caption*{Evolution of wealth $V^*_t$.}
  \end{minipage}

  \vspace{1em}

  \begin{minipage}[b]{0.72\textwidth}
      \centering
      \includegraphics[width=\linewidth]{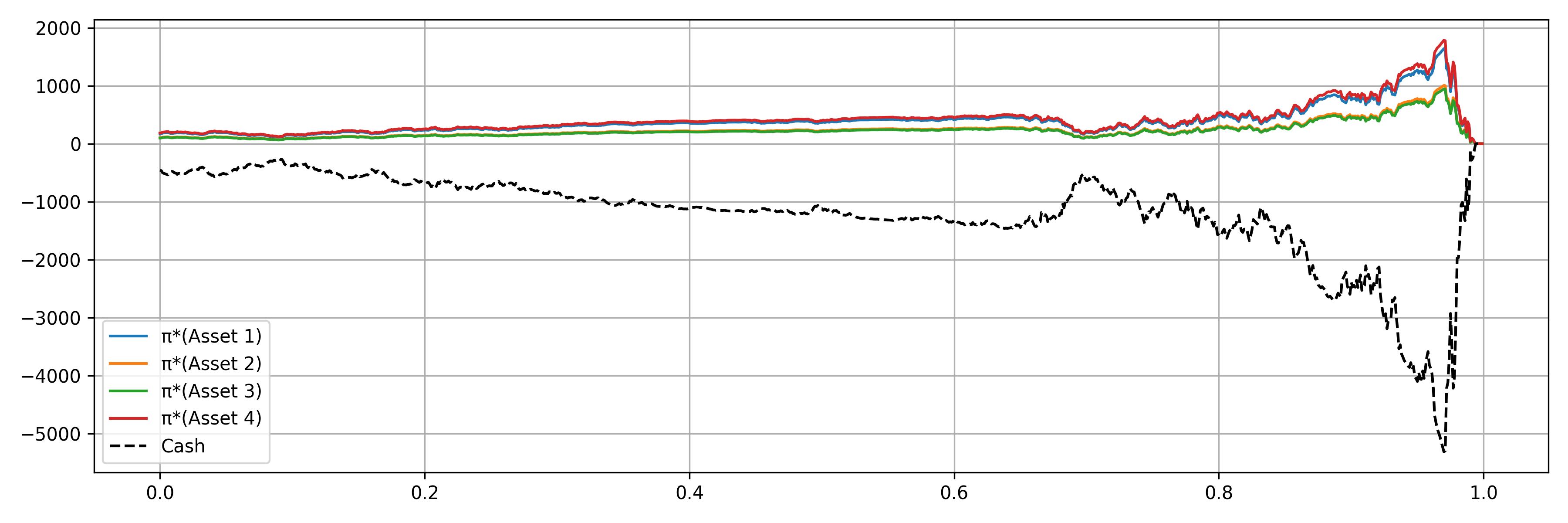}
      \caption*{Position value $\beta_t \odot S_t (=\pi_t^*)$ and cash position.}
  \end{minipage}

  \caption{Portfolio dynamics for seed 136 with $V^*_T = 0$}
  \label{fig:combined_K30_seed136}
\end{figure}

\subsubsection*{Analysis of the Distribution of $V_T^*$}

As expected, by Proposition~\ref{prop:sol_lagg_relaxation_pb}, $V_T^*$ takes only three values: $B$, $-\alpha^*$, and $0$. For $B=500$, we observe approximately $5.5\%$ at $B$, $0.08\%$ at $0$, and the remainder at $-\alpha^*$.

As shown in the previous figure, the three distinct cases described earlier can be clearly observed.

We emphasize that implementing the optimal control involves substantial risk in every solution. In particular, the optimal policy uses a very large short cash position. For example, in Figure~\ref{fig:combined_K30_seed2} the cash position reaches approximately $-100\times$ the initial wealth.

We performed a Monte Carlo simulation of the final portfolio value with $500{,}000$ paths and recovered the expected values for the Value-at-Risk (VaR), the mean, and the (Differential) Conditional Value-at-Risk (DCVaR/CVaR).

\subsubsection*{Discussion of the Impact of $B$ in Practice}

$B$ is the upper bound on the portfolio value introduced to make the optimization problem well-posed. It plays a crucial role in shaping the optimal strategy. If $B$ is set too low, it restricts the portfolio, causing the optimal strategy to hit the cap frequently and reducing $\E[V_T^*]$. Conversely, if $B$ is set too high, the optimal strategy tends to exploit a small number of rare but highly profitable paths to boost the expected terminal value. This brings us back to the earlier issue, in an even more pronounced form, where only a very small fraction of trajectories yield significant gains, while the vast majority remain conservative.

We observed that the expected terminal value is not significantly affected by $B$ once it is set to a reasonably high level, starting from $B=140$ in our experiments (see Figure~\ref{fig:v_vs_B_K30}).

\begin{figure}[H]
  \centering
  \includegraphics[width=0.65\textwidth]{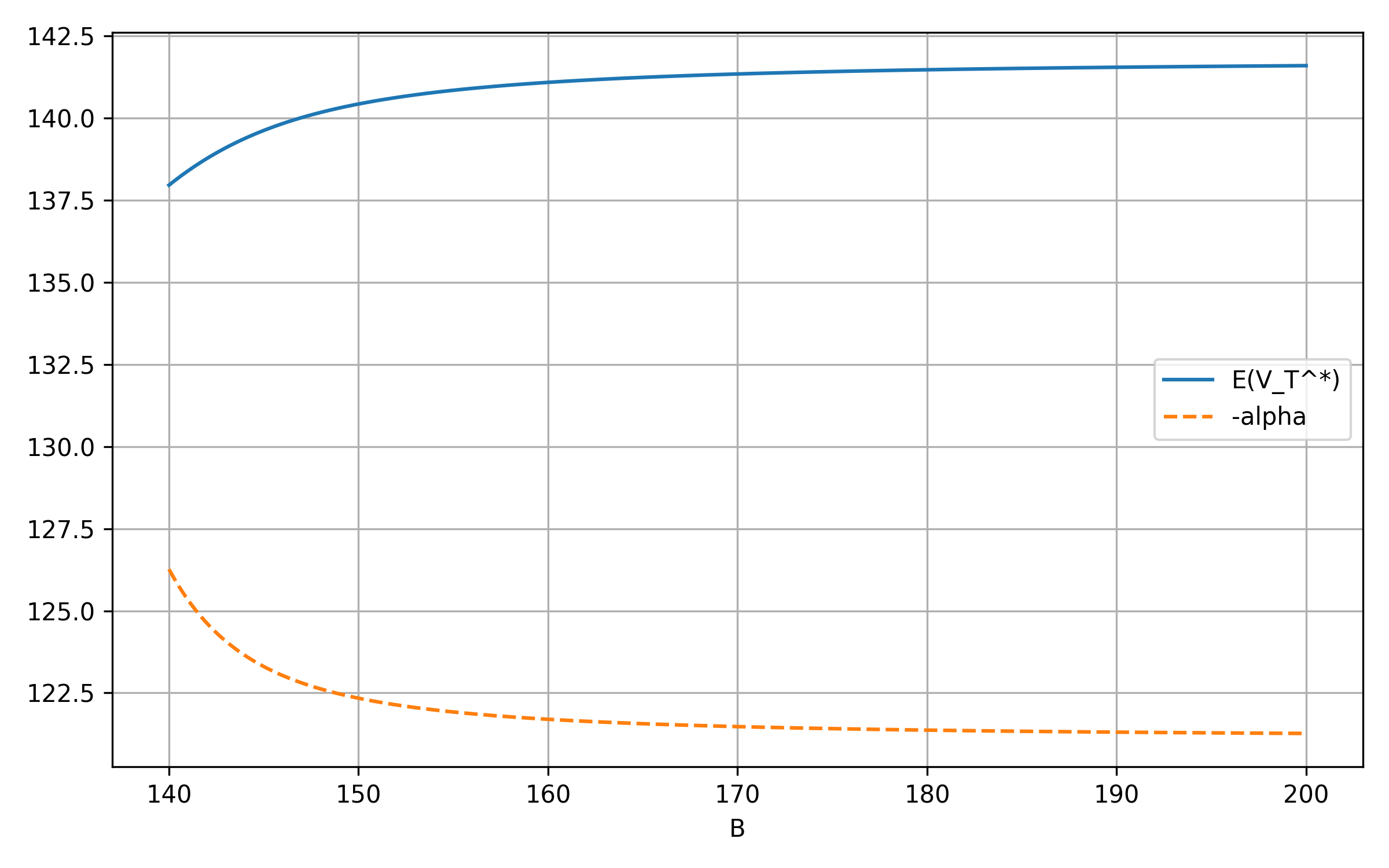}
  \caption{Optimal expected terminal value $\E[V^*_T]$ and $-\alpha^*$ as functions of $B$.}
  \label{fig:v_vs_B_K30}
\end{figure}

This observation helps estimate the probabilities of the three scenarios as functions of $B$. 
When the DCVaR constraint is met with equality,
\begin{align*}
   \E\!\left[\,V_T^* + \alpha^* + (1-\kappa)^{-1}\,(-V_T^*-\alpha^*)^+ \right] &= K,\\
   \E[V_T^*] + \alpha^* + (1-\kappa)^{-1}\bigl(-\alpha^*\,\P_0\bigr) &= K,
\end{align*}
hence
\[
   \P_0 \;=\; \frac{\bigl(K - \E[V_T^*]-\alpha^*\bigr)(1-\kappa)}{-\alpha}.
\]

We also have
\begin{align*}
  \E[V_T^*] &= B\,\P_B - \alpha^*\,\P_\alpha,\\
  \E[V_T^*] &= B\,\P_B - \alpha^*\bigl(1-\P_B-\P_0\bigr),
\end{align*}
and therefore
\[
  \P_B \;=\; \frac{\E[V_T^*]+\alpha^*\,(1-\P_0)}{B+\alpha^*}.
\]

Under the empirical observation that $\E[V_T^*]$ stabilizes for sufficiently large $B$, these relations determine the evolution of $\P_B$ and $\P_\alpha$ with $B$ once $\P_{0}$ is fixed. In our case, $\P_{0}=0.000772$.

\begin{figure}[H]
  \centering
  \begin{minipage}[b]{0.49\textwidth}
      \centering
      \includegraphics[width=\linewidth]{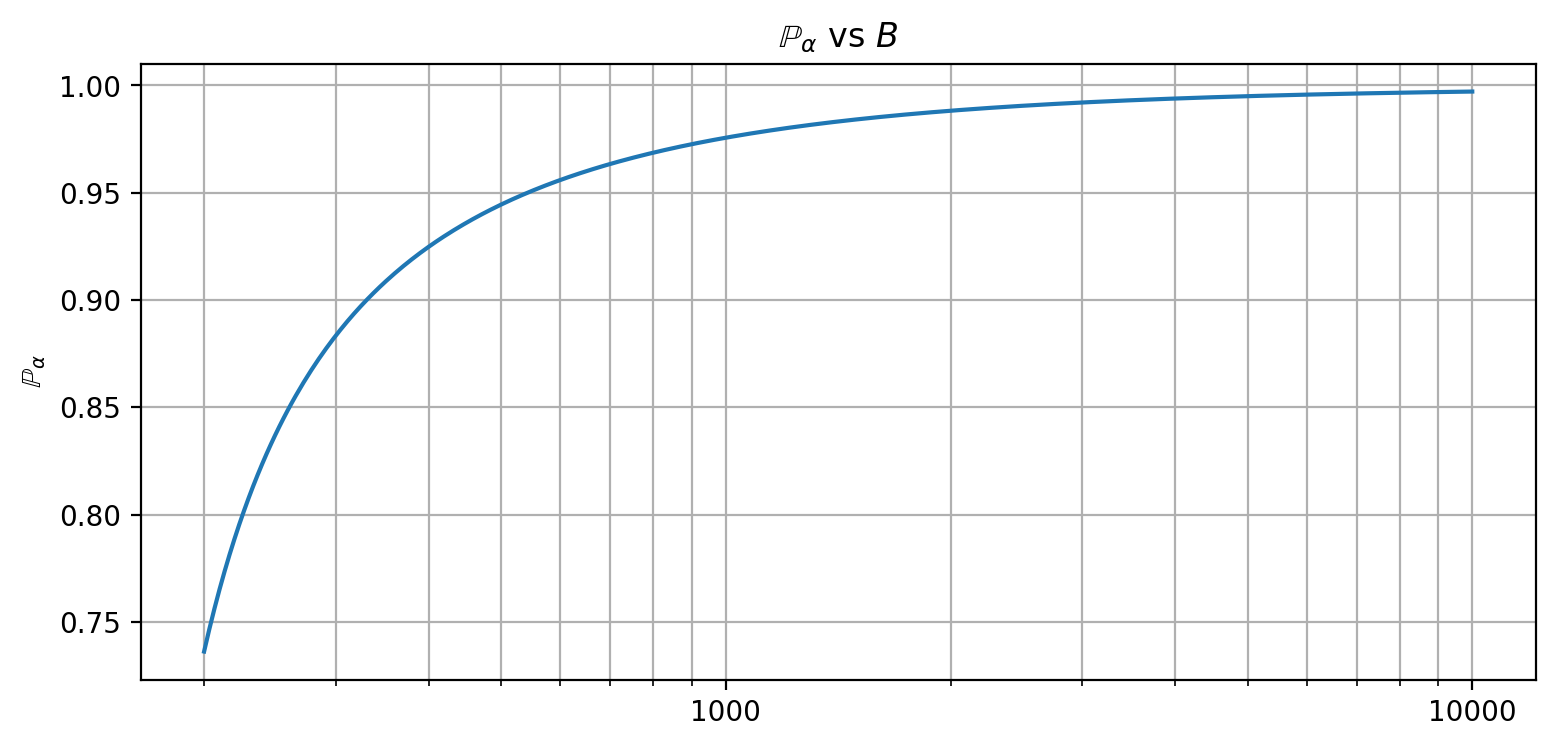}
  \end{minipage}\hfill
  \begin{minipage}[b]{0.49\textwidth}
      \centering
      \includegraphics[width=\linewidth]{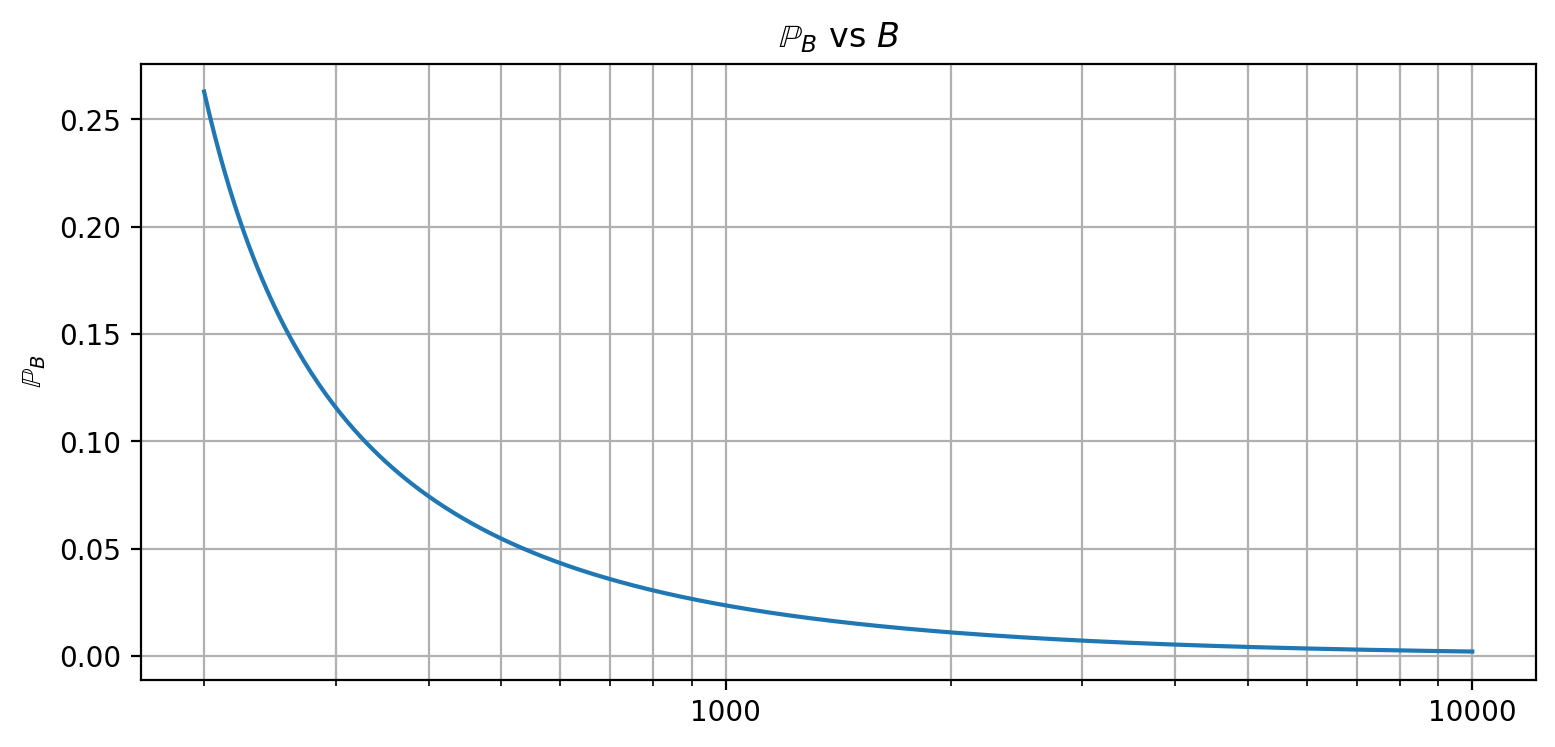}
  \end{minipage}
  \caption{$\P_B$ and $\P_\alpha$ as functions of $B$ for $B\in[200,10\,000]$.}
  \label{fig:palpha_pbeta}
\end{figure}

As shown above, increasing $B$ does not sufficiently improve $\E[V_T^*]$. It is therefore preferable to choose a lower $B$ to avoid a variance explosion and to reduce the concentration of gains in a few scenarios.

Such a strategy is difficult to justify in traditional asset management, as it concentrates most outcomes around the threshold $-\alpha^*$ and leaves only a very small fraction of scenarios achieving significant gains.

\subsection{Efficent frontier visualisation}

We now want to visualize the evolution of $\E[V^*_T]$ and $-\alpha^*$ with $K$. 

\begin{figure}[H]
  \centering
  \includegraphics[width=0.65\textwidth]{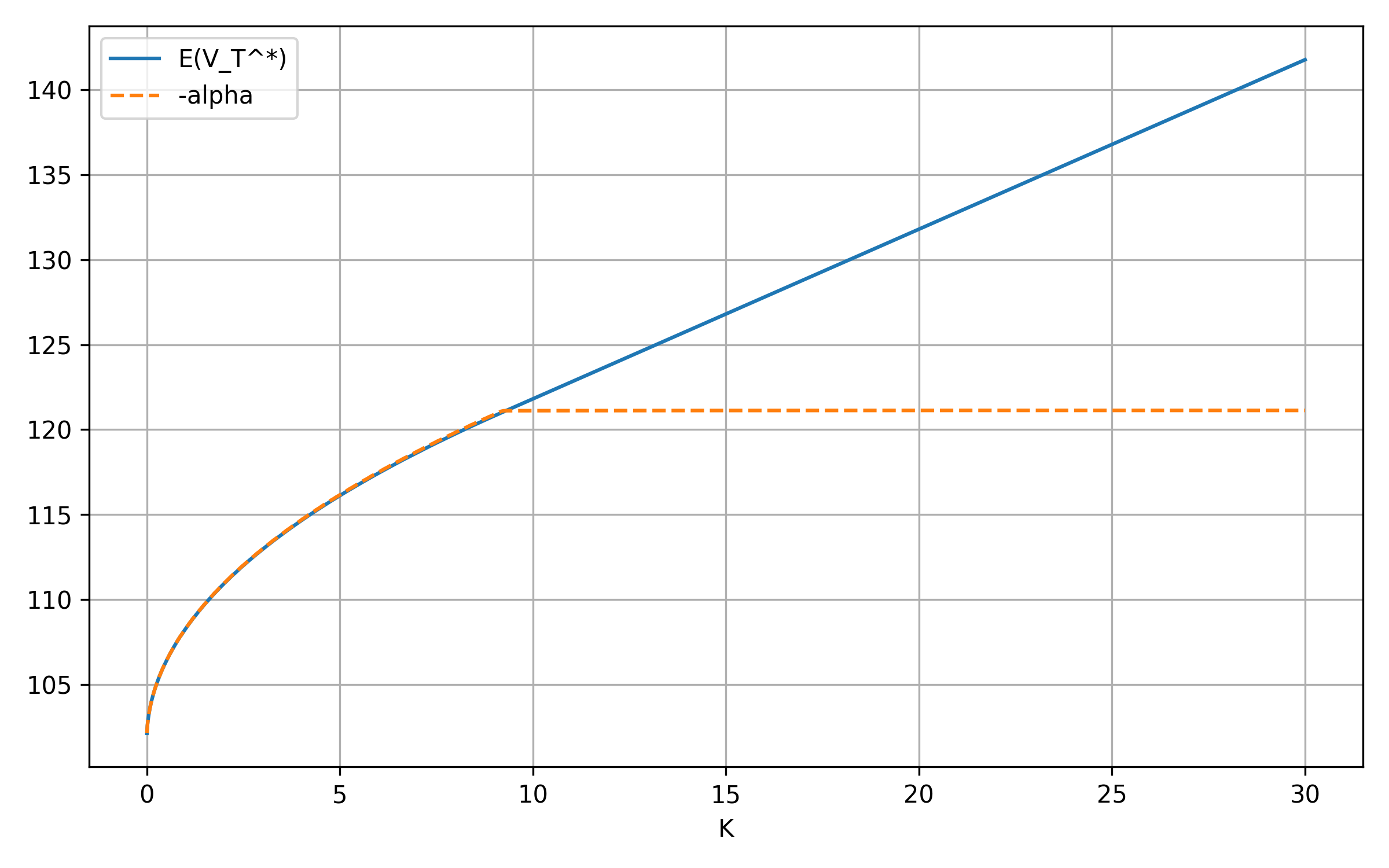}
  \caption{Optimal expected terminal value $\E[V^*_T]$ and $-\alpha^*$ as functions of $K$ for the previous case with four risky assets and $B=500$.}
  \label{fig:v_vs_k_K30}
\end{figure}

As shown in Figure~\ref{fig:v_vs_k_K30}, we observe two distinct regimes in the evolution. The first occurs between $K = 0$ and $K = 8$, where the behavior is concave. The second regime begins after $K = 8$ and exhibits a linear evolution. With a regression we found that the slope of this second part is $1$. This behavior is expected, as in the second regime the optimal strategy primarily scales with the level of risk, resembling the structure of the classical Markowitz case.\\

\begin{rem}[On the regime change]
For a fixed $K$ and $\alpha \in (\underline{\alpha}(K), \overline{\alpha}(K))$, we first consider the optimal objective value given by equation~\eqref{thm_lpm_dq<1_eq1}:
\begin{align*}
E[V_T^*](\alpha) 
&= (B+\alpha)\Phi\big(Y_1(0)\big)-\alpha\Phi\big(Y_2(0)\big) \\[6pt]
&= K - \alpha \left( 1 - \frac{1}{1 - \kappa}\big(1 - \Phi(Y_2(0))\big) \right).\\
&= K -\alpha \left( 1 - \frac{1}{1 - \kappa} \left( 1 - H_0\left( H_1^{-1}\left( \frac{w_0 - (B + \alpha) H_1(\delta(\alpha))}{-\alpha} \right) \right) \right) \right)
\end{align*}

The dependence on $\alpha$ arises from two competing effects: the linear factor $-\alpha$, which decreases monotonically, and a nonlinear term that may decrease over some range and increase over another. This trade-off creates two regimes: for small $K$, the maximum occurs at the boundary $\underline{\alpha}(K)$, while for larger $K$ it admits an interior solution determined by the first-order condition.
\end{rem}

\section{Summary and limitations}
Two major points can be drawn from the results presented here. First, our solution method does not allow for constraints on the portfolio weights, which leads to an unconstrained, often excessive, borrowing in the optimal strategy. In continuous time, this is partially offset by the fact that the strategy can be adjusted at every instant, but in practice, discrete rebalancing and transaction costs would render such an approach highly risky. The lack of position limits is therefore a serious drawback in our use case.

Second, we find that the final expected portfolio value is relatively insensitive to the parameter \(B\), yet the distribution of outcomes can vary dramatically. This is problematic because even a marginal improvement in expected return can lead to a radically different strategy. It is hence unsurprising that many authors introduce a variance based penalty term to smooth out this “lottery like” behavior.

\appendix

\section{Proof of Proposition \ref{prop:sol_lagg_relaxation_pb}}
The objective of this appendix section is to prove Proposition~\ref{prop:sol_lagg_relaxation_pb} through a standard analysis of the function depending on the variable \(Z\) and the Lagrange multipliers.

\begin{proof}
  After removing the constant terms in \eqref{eq:lagg_relaxation_pb}, we focus on the inner problem:
\begin{equation}
  \min_{M \in \mathbb{R}, \  0 \leq M \leq B} -M + \lambda \left( M + (1 - \kappa)^{-1} (-M-\alpha)^+ \right) + \eta Z M. 
  \label{inner_problem}
\end{equation}

We solve when:
\begin{equation*}
    -M - \alpha \geq 0 \ \ \text{and} \ -M - \alpha \leq 0,
\end{equation*}
and recall that we have assumed $-\alpha \geq 0$ and established that $\eta > 0$.

If $-M - \alpha \geq 0$, then \eqref{inner_problem} becomes:
\begin{equation*}
  \min_{0 \leq M \leq -\alpha} -M + \lambda \left( M + (1 - \kappa)^{-1} (-M-\alpha) \right) + \eta Z M. 
\end{equation*}

By the earlier assumption \(\kappa>0.5\), it follows that $(1-\kappa)^{-1}-1 > 0$.

$g:M \mapsto -M + \lambda \left( M + (1 - \kappa)^{-1} (-M-\alpha) \right) + \eta Z M$ is affine and we can determine its direction of variation with the sign of $-1 + \lambda \left( 1 - (1 - \kappa)^{-1} \right) + \eta Z$

Thus,
\begin{align*}
    g(M^*) = \begin{cases}
        -\alpha \lambda (1 - \kappa)^{-1}, & \text{if } \frac{\lambda\left((1-\kappa)^{-1} -1 \right)+1}{\eta} < Z.\\
        -\alpha \left(-1 + \lambda + \eta Z \right), & \text{if }  Z \leq \frac{\lambda\left((1-\kappa)^{-1} -1 \right)+1}{\eta}.
    \end{cases}
\end{align*}

If $-M - \alpha \leq 0$, then \eqref{inner_problem}  becomes:
\begin{equation*}
  \min_{-\alpha \leq M \leq B} -M + \lambda M + \eta Z M.
\end{equation*}

At this stage, it becomes necessary to have an upper bound $B$ on the random variable $M$, and consequently on $V_T$. In the subsequent analysis, we examine the behavior of the solution as $B$ tends to infinity.

With the same reasonning on $g:M \mapsto -M + \lambda M + \eta Z M$ we obtain:

\begin{align*}
    g(M^*) = \begin{cases}
        -\alpha(-1 + \lambda + \eta Z), & \text{if } \frac{1-\lambda}{\eta} \leq Z \\
        B (-1 + \lambda + \eta Z), & \text{if } Z < \frac{1-\lambda}{\eta}.
    \end{cases}
\end{align*}

Finally, there are four possible configurations. In three of them, the minimum of $g$ over the interval $M \in [0, B]$ is straightforward to identify, namely when the function is increasing/increasing, decreasing/decreasing, or decreasing/increasing. The remaining case occurs when the function is increasing then decreasing. This situation occurs when:

\[
\frac{\lambda\left((1 - \kappa)^{\,-1} - 1\right) + 1}{\eta} < Z \leq \frac{1 - \lambda}{\eta}
\]

However, this inequality is never satisfied. Indeed, observe that:

\[
\frac{\lambda\left((1 - \kappa)^{\,-1} - 1\right) + 1}{\eta} = \frac{1 - \lambda}{\eta} + \frac{\lambda (1 - \kappa)^{\,-1}}{\eta} > \frac{1 - \lambda}{\eta}
\]

Finally the optimal solution is given by:
\begin{align*}
    M^* = \begin{cases}
      B, & \text{if } \ Z \leq \frac{1-\lambda}{\eta} \\
      -\alpha, & \text{if } \ \frac{1-\lambda}{\eta} < Z \leq \frac{\lambda\left((1-\kappa)^{-1} -1 \right)+1}{\eta} \\
      0, & \text{if } \ \frac{\lambda\left((1-\kappa)^{-1} -1 \right)+1}{\eta} < Z.
    \end{cases}
\end{align*}
\end{proof}

\section{Proof of Theorem \ref{th_lagrange_mul_behavior}}

The objective of this appendix section is to characterize the feasibility set for the parameters $\delta$ and $\rho$, which are auxiliary variables associated with the Lagrange multipliers $\lambda$ and $\eta$, ensuring the validity of \eqref{eq:consx0_new}. This feasibility set enables us to identify the attainable maximum and minimum values of the parameter $K$ in \eqref{eq:consK_new}. Moreover, the analysis provides insights into the behavior of the Lagrange multipliers as functions of $K$, and serves as the basis for proving Theorem~\ref{th_lagrange_mul_behavior}.

Recall that, for $p = 0,1$: $H_p(y) = \E\left[ (Z)^p \ind2{\{Z \leq y \}} \right]$ and define
\begin{equation}
  \delta = \frac{1-\lambda}{\eta} \ \ , \ \ \rho = \frac{\lambda ((1-\kappa)^{-1})}{\eta}.\label{eq_lambda_eta_mapping}
\end{equation}
Clearly, $\rho > 0$. We impose $\delta \geq 0$, and defer a detailed discussion of this assumption to a later remark.

we rewrite \eqref{eq:consx0_new},
\begin{equation}
  I(\delta, \rho) = w_0, \label{eq:consx0_I}
\end{equation}
with $I(\delta, \rho) = B H_1(\delta) - \alpha \left( H_1(\delta + \rho) - H_1(\delta) \right)$, thanks to
\begin{align*}
  \begin{cases}
    C_1 =& \{ Z \leq \frac{1-\lambda}{\eta} = \delta \},\\
    C_3 =& \{ Z > \frac{\lambda\left((1-\kappa)^{-1} -1 \right)+1}{\eta} = \delta + \rho \},\\
    \ind2{C_2} =&\ind2{\overline{C_3}}-\ind2{C_1}.
  \end{cases}
\end{align*}

We also have $P_1 = H_0(\delta), \ P_2 = H_0(\delta + \rho) - H_0(\delta)$ and $P_3 = 1- H_0(\delta + \rho)$ so we can rewrite \eqref{eq:consK_new} (in the equal case),
\begin{align}
  & B H_0(\delta) - \alpha \left( H_0(\delta + \rho) - H_0(\delta) \right) -\alpha(1-\kappa)^{-1}(1-H_0(\delta + \rho)) = K - \alpha. \nonumber \\ 
  \Leftrightarrow &(B+\alpha) H_0(\delta) - \alpha \left( 1-(1-\kappa)^{-1} \right)(H_0(\delta+\rho)-1) = K. \label{eq:consK_H}
\end{align}

\begin{lemma}
  $I(\delta, \rho)$ is monotonically increasing with respect to $\delta$ and $\rho$, respectively. \label{lem:I_monotonicity}
\end{lemma}
\begin{proof}
  For $\delta_1 > \delta_2 > 0$ and $\rho>0$:
  \begin{equation*}
    I(\delta_1, \rho) - I(\delta_2, \rho) = (B+\alpha)(H_1(\delta_1)-H_1(\delta_2)) - \alpha(H_1(\delta_1+\rho)-H_1(\delta_2+\rho)) > 0,
  \end{equation*}
  because $H_1$ is monotonically increasing and $B+\alpha > 0$ and $-\alpha >0$.\\
  For $\rho_1 > \rho_2 > 0$ and $\delta>0$ we have the same result.
\end{proof}

Rearranging \eqref{eq:consx0_I} gives 
\begin{equation}
  H_1(\delta+\rho) = \frac{w_0 -(B+\alpha)H_1(\delta)}{-\alpha}.
  \label{eq:H1_carac}
\end{equation}

The expression \eqref{eq:H1_carac} provides us an upper and lower limits of $\delta$ i.e.

\begin{align*}
  \overline{\delta} &:= \sup\{\delta \in \R \ | \ I(\delta, \rho) = w_0, \delta > 0, \rho > 0 \}, \\
  \underline{\delta} &:= \inf\{\delta \in \R \ | \ I(\delta, \rho) = w_0, \delta > 0, \rho > 0 \}.
\end{align*}

Due to the monotonicity of $H_1(\cdot)$, we have $0\leq H_1(\delta)< H_1(\delta+\rho)< H_1(\infty)=\E[Z]$, which further leads to
\begin{align}
H_1(\delta)< \frac{w_0-(B+\alpha)H_1(\delta)}{-\alpha}<\E[Z],~~\Rightarrow~
~\frac{w_0+\alpha \E[Z]}{B+\alpha}<H_1(\delta)\leq \frac{w_0}{B}. \label{ineq:low_upp_delta}
\end{align}

The inequality in (\ref{ineq:low_upp_delta}) provides the lower and  upper limits of $\delta$ as follows,
\begin{align}
\underline{\delta}&=
\begin{dcases}
0 &\hbox{if} ~~w_0< -\alpha \E[Z],\\
H_1^{-1}\left(\frac{w_0+\alpha \E[Z]}{B+\alpha}\right) &\hbox{if}~~w_0\geq - \alpha \E[Z],
\end{dcases}\label{eq:delta_lower}\\
\bar{\delta}&=H_1^{-1}(w_0/B). \label{eq:delta_upper}
\end{align}

For $\delta \in [\underline{\delta},\bar{\delta}]$, \eqref{eq:H1_carac} holds, substituting $\delta+\rho$ in (\ref{eq:H1_carac}) back to (\ref{eq:consK_H}) yields the following function $L(\delta)$:
\begin{align}
    L(\delta):= (B+\alpha)H_0(\delta)-\alpha(1-(1-\kappa)^{-1})\left( H_0\left(H_1^{-1}\left(\frac{w_0-(B+\alpha)H_1(\delta)}{-\alpha} \right)\right)-1\right). \label{eq:defL}
\end{align}

\begin{lemma}
  $L(\delta)$ is a monotonically increasing function with respect to $\delta$ between $[\underline{\delta},\bar{\delta}]$. \label{lem:L_increasing}
\end{lemma}

\begin{proof}
  Let $\underline{\delta} < \delta_2 < \delta_1 <\bar{\delta}$, we have for $i=1,2$:
  \begin{equation*}
    L(\delta_i):= (B+\alpha)H_0(\delta_i)-\alpha(1-(1-\kappa)^{-1}) \left(H_0\left(H_1^{-1}\left(\frac{w_0-(B+\alpha)H_1(\delta_i)}{-\alpha} \right)\right)-1\right).
  \end{equation*}
  Let $\rho_1$ and $\rho_2$ satisfy 
  \begin{align}
  \delta_i+\rho_i&=H_1^{-1}\left(\frac{w_0-(B+\alpha)H_1(\delta_i)}{-\alpha}\right),~~
  \Rightarrow~
  (B+\alpha)H_1(\delta_i)-\alpha H_1(\delta_i+\rho_i)=w_0. \label{eq:delta_rho}
  \end{align}
  These values exist due to Lemma \ref{lem:I_monotonicity} and that $\underline{\delta} < \delta_i <\bar{\delta}$. Moreover, $H_1$ and $H_1^{-1}$ are monotonically increasing, so using \eqref{eq:delta_rho}
  \begin{align*}
    H_1(\delta_2) < H_1(\delta_1) & \Rightarrow~ \frac{w_0-(B+\alpha)H_1(\delta_2)}{-\alpha} < \frac{w_0-(B+\alpha)H_1(\delta_1)}{-\alpha},\\
    & \Rightarrow~ \delta_1 + \rho_1 < \delta_2 + \rho_2. \\
  \end{align*}
  let us study the sign of 
  \begin{align*}
    L(\delta_1)-L(\delta_2) &= (B+\alpha)(H_0(\delta_1) - H_0(\delta_2))) - \alpha(1-(1-\kappa)^{-1})\left( H_0(\delta_1 + \rho_1) - H_0(\delta_2 + \rho_2)\right),\\
    &= (B+\alpha)(H_0(\delta_1) - H_0(\delta_2))) - \alpha((1-\kappa)^{-1}-1)\left( H_0(\delta_2 + \rho_2) - H_0(\delta_1 + \rho_1)\right) > 0,
  \end{align*}
  which concludes the proof.
\end{proof}

Lemma~\ref{lem:L_increasing} characterizes the admissible values of $L$, and thus of $K$, through the limits:
$$
\lim_{\delta \to \underline{\delta}} L(\delta) = \underline{K}, \quad \lim_{\delta \to \bar{\delta}} L(\delta) = \overline{K}.
$$

\begin{rem}
In practice, we impose a lower bound of zero on $\delta$, but it is theoretically possible for $\delta$ to take negative values. When $\delta < 0$, it does not affect the value of $I(\delta, \rho)$, and everything behaves as if $\delta = 0$, since $H_p(\delta) = 0$ in that case. Consequently, the final result is not affected either, as it depends on $H_p(\delta)$, $H_p(\delta+\rho)$, and $\rho$, and is therefore entirely determined by $H_1(\delta)$, as shown in equation~\eqref{eq:H1_carac}. Finally, when $\delta$ may become negative, we set $\delta^* = 0$ and therefore take $\lambda = 1$, even though there may exist infinitely many positive pairs $(\lambda, \eta)$ that satisfy the unique optimal strategy (case (iii) in each theorem).
\end{rem}

Finally we can define $\underline{K}$ and $\overline{K}$ using \eqref{eq:consK_H}
\begin{align*}
  &\underline{K}=\begin{dcases}
  -\alpha (1-(1-\kappa)^{-1}) \left( H_0\left(  H_1^{-1}\left(\frac{w_0}{-\alpha}\right)\right)-1\right)& \hbox{if}~~w_0<-\alpha \E[Z], \\
  (B+\alpha)H_0\left( H_1^{-1}\left( \frac{w_0+\alpha \E[Z]}{B+\alpha}\right)\right) & \hbox{if} ~~w_0\geq -\alpha \E[Z],
  \end{dcases}\\
  &\overline{K}=B H_0\left(  H_1^{-1}(w_0/B)\right) -  \alpha\left( (1-\kappa)^{-1} \left( 1 - H_0\left(  H_1^{-1}(w_0/B)\right) \right) -1\right).
\end{align*}

Proof of Theorem \ref{th_lagrange_mul_behavior}:
\begin{proof}
  If $\underline{K} < K < \overline{K}$, then by the monotonicity of $L$ from Lemma \ref{lem:L_increasing}, there exists a unique $\underline{\delta} < \delta^* < \bar{\delta}$ such that $L(\delta^*) = K$. Furthermore, due to the monotonicity of $I$ from Lemma \ref{lem:I_monotonicity}, a unique $\rho^*$ can be determined by substituting $\delta^*$ into \eqref{eq:H1_carac}. Moreover, the pairs $(\delta^*, \rho^*)$ and $(\lambda, \eta)$ form a one-to-one correspondence, as established in \eqref{eq_lambda_eta_mapping}, thereby completing the proof for case (i).

  In case (ii), since \( \underline{K} \) is equal to zero, meaning that the \( \mathrm{DCVaR} \) is null, any portfolio containing risky assets cannot attain a null \( \mathrm{DCVaR} \). Consequently, this configuration is not feasible, and the portfolio must be composed entirely of the risk-free asset.

  In case (iii), \( K = \underline{K} \) and \( \omega_0 < -\alpha \mathbb{E}[Z] \), then it follows that \( \frac{1 - \lambda}{\eta} = \delta^* = \underline{\delta} = 0 \), which implies \( \lambda = 1 \). Substituting into equation~\eqref{eq:consx0_I}, we deduce that $ \rho^* = H_1^{-1} \left( \frac{\omega_0}{-\alpha} \right)$ and $\eta = \frac{1}{(1 - \kappa) H_1^{-1} \left( \frac{\omega_0}{-\alpha} \right)}$.

  In case (iv), when $w_0 > -\alpha E[Z]$, equations~\eqref{eq:H1_carac} and~\eqref{eq:delta_lower} yield 
  $\delta + \rho = +\infty$ and $\delta = H^{-1}\!\left(\frac{w_0 + \alpha E[Z]}{B+\alpha}\right) > 0$. In this limiting case, as $K \to \underline{K}$, we have $\tfrac{1}{\eta} \to +\infty$, $\lambda \to 1^{-}$, and $\frac{1-\lambda}{\eta} \;\to\; H^{-1}\!\left(\frac{w_0 + \alpha E[Z]}{B+\alpha}\right)$,
  which completes the proof for case (iv).

  For (v), from \eqref{eq:H1_carac} and \eqref{eq:delta_lower}, we know that the system of two equations in \eqref{c_ZV_equal_V0} and \eqref{c_DCVaR_K} hold when $\rho=\bar{\rho}=0$ with the correspondent $\delta=\bar{\delta}= H^{-1}_1(w_0/B)$ by \eqref{eq:delta_rho}. Due to the one-to-one mapping in \eqref{eq_lambda_eta_mapping}, we have $\lambda=0$ and $\eta^*=1/H_1^{-1}(w_0/B)$.
\end{proof}

\section{Auxiliary Lemmas}
The objective of this appendix section is to present two technical lemmas: the first supports Theorem~\ref{thm_solution} by computing an exact expectation, and the second establishes the convexity of the solution to Equation~\eqref{pb_P_DCVAR_equivalent} with respect to \(\alpha\).

\begin{lemma}\label{lem_truncated_expectation}
Let $Y$ be a random variable that follows the normal distribution with mean $\mu$ and variance $v^2$, respectively. Then, we have
\begin{align*}
\E[e^{aY}\cdot \1_{Y \leq d}]&=\exp\left(a\mu+\frac{a^2v^2}{2}\right) \Phi\left(\frac{d-\mu}{v}-av\right),
\end{align*}
where $\Phi(\cdot)$ is the cumulative distribution function of standard normal random variable.
\end{lemma}
\begin{proof}
  Let $Z=(Y-\mu)/v$. Then $Z$ follows the standard normal distribution and
\begin{align*}
\E[e^{aY}\1_{Y\leq d}]&=\E[e^{a(zv+\mu)}\1_{zv+\mu\leq d} ],\\
&=\frac{1}{\sqrt{2\pi}}\int_{-\infty}^{\frac{d-\mu}{v}} \exp\left(-\frac{(z^2-2azv-2a\mu}{2}\right) dz,\\
&=\frac{1}{\sqrt{2 \pi}} \exp(\frac{2a\mu+a^2v^2}{2})\int_{-\infty}^{\frac{d-\mu}{v}} \exp \left(-\frac{(z-av)^2}{2}  \right)dz,\\
&=\exp\left(a\mu+\frac{a^2v^2}{2}\right) \Phi\left( \frac{d-\mu}{v}-av\right),
\end{align*}
which is exactly (\ref{lem_truncated_expectation}).
\end{proof}\\

\begin{lemma}\label{lem_result_is_convex}
  Let $f: \H^2 \to \mathbb{R}$ and $g: \H^2 \times I \to \mathbb{R}$ be two convex functions, where $I \subset \mathbb{R}$. Let $K \in \mathbb{R}$ be such that
  $$
  K \geq \inf_{(u,\alpha) \in \H^2 \times I} g(u,\alpha).
  $$
 And, for a fixed $\alpha \in \mathbb{R}$, define $R(\alpha)$ as
  \begin{align*}
    R(\alpha) := 
    \begin{dcases}
      \inf_{u \in \H^2} \left\{ f(u) \,\middle|\, g(u,\alpha) \leq K \right\}, & \quad \text{if } K \geq \inf_{u \in \H^2} g(u,\alpha), \\
      +\infty, & \quad \text{otherwise}.
    \end{dcases}
  \end{align*}
  Then the function $f: \alpha \mapsto R(\alpha)$ is convex on $\mathbb{R}$.
  
  \begin{proof}
    We aim to prove the convexity of the function $f: \alpha \mapsto R(\alpha)$ on the set
    $$
    A = \left\{ \alpha \in \mathbb{R} \,\middle|\, \inf_{u \in \H^2} g(u,\alpha) \leq K \right\}.
    $$
    Let $\alpha_1, \alpha_2 \in A$, and let $\lambda \in [0,1]$. Define the convex combination
    $$
    \alpha_\lambda := \lambda \alpha_1 + (1 - \lambda) \alpha_2.
    $$
    Since $A$ is convex (as a sublevel set of a convex function), we have $\alpha_\lambda \in A$.
    
    Let $u_1, u_2 \in \H^2$ be $\varepsilon$-optimal solutions for $\alpha_1$ and $\alpha_2$, respectively, such that
    \begin{align*}
      g(u_1, \alpha_1) &\leq K, \quad f(u_1) \leq R(\alpha_1) + \varepsilon, \\
      g(u_2, \alpha_2) &\leq K, \quad f(u_2) \leq R(\alpha_2) + \varepsilon.
    \end{align*}
    Define the convex combination $u_\lambda := \lambda u_1 + (1 - \lambda) u_2 \in \H^2$. By convexity of $f$ and $g$, we have
    \begin{align*}
      f(u_\lambda) &\leq \lambda f(u_1) + (1 - \lambda) f(u_2), \\
      g(u_\lambda, \alpha_\lambda) &\leq \lambda g(u_1, \alpha_1) + (1 - \lambda) g(u_2, \alpha_2) \leq K.
    \end{align*}
    Hence, $u_\lambda$ is a feasible point for $R(\alpha_\lambda)$, and it follows that
    $$
    R(\alpha_\lambda) \leq f(u_\lambda) \leq \lambda f(u_1) + (1 - \lambda) f(u_2) \leq \lambda R(\alpha_1) + (1 - \lambda) R(\alpha_2) + \varepsilon.
    $$
    Since $\varepsilon > 0$ is arbitrary, we conclude that
    $$
    R(\alpha_\lambda) \leq \lambda R(\alpha_1) + (1 - \lambda) R(\alpha_2),
    $$
    which proves the convexity of $R$ on $A$.
  \end{proof}
\end{lemma}

\printbibliography

\end{document}